\documentclass[a4paper]{article}

\usepackage{
amsmath,
amsthm,
amscd,
amssymb,
}
\usepackage{comment}
\usepackage{tikz}
\usetikzlibrary{positioning,arrows,calc}
\usepackage{xspace}
\usepackage{wasysym}

\usepackage{graphicx}

\usepackage{url}

\theoremstyle{plain}
\newtheorem{lemma}{Lemma}
\newtheorem{definition}{Definition}

\newtheorem{theorem}{Theorem}

\newtheorem{question}{Question}
\newtheorem{remark}{Remark}

\newtheoremstyle{derp}
{3pt}
{3pt}
{}
{}
{\upshape}
{:}
{.5em}
{}
\theoremstyle{derp}
\newtheorem{example}{Example}

\newcommand{\R}{\mathbb{R}}

\newcommand{\Z}{\mathbb{Z}}
\newcommand{\C}{\mathcal{C}}
\newcommand{\D}{\mathcal{D}}
\newcommand{\N}{\mathbb{N}}

\newcommand{\dist}{\mathrm{d}}

\newcommand{\PP}{\mathcal{P}}

\newcommand{\CHull}{\mathrm{CHull}}

\newcommand\xqed[1]{%
  \leavevmode\unskip\penalty9999 \hbox{}\nobreak\hfill
  \quad\hbox{#1}}
\newcommand\qee{\xqed{$\fullmoon$}}

\newcommand{\pow}{\mathcal{P}}
\newcommand{\follow}{\mathcal{F}}

\newcommand{\GL}{\mathrm{GL}}
\newcommand{\SSS}{\mathcal{S}}

\title{Avoshifts, Unishifts and Nondeterministic Cellular Automata}

\author{
Ville Salo \\
vosalo@utu.fi
}

\begin{document}
\maketitle

\begin{abstract}
In this paper, we study avoshifts and unishifts on $\Z^d$. Avoshifts are subshifts where for each convex set $C$, and each vector $v$ such that $C \cup  \{\vec v\}$ is also convex, the set of valid extensions of globally valid patterns on $C$ to ones on $C \cup \{v\}$ is determined by a bounded subpattern of $C$. Unishifts are the subshifts where for such $C, \vec v$, every $C$-pattern has the same number of $\vec v$-extensions. Cellwise quasigroup shifts (including group shifts) and TEP subshifts are examples of unishifts, while unishifts and subshifts with topological strong spatial mixing are examples of avoshifts.
We prove that every avoshift is the spacetime subshift of a nondeterministic cellular automaton on an avoshift of lower dimension up to a linear transformation and a convex blocking. From this, we deduce that all avoshifts contain periodic points, and that unishifts have dense periodic points and admit equal entropy full shift factors.
\end{abstract}


\section{Introduction}


It is a common problem in symbolic dynamics to try to understand to what extent we can locally tell which patterns end up being globally valid. In this paper, we study avoshifts and unishifts, where it is taken as a defining property that local validity (in suitable situations) implies global validity. We show that this already imposes a lot of structure on the subshift.

Avoshifts and unishifts were first defined in \cite{Sa24} (the latter under the term ``equal extension counts''), where it was shown that avoshifts (and therefore the subclass of unishifts) have good algorithmic properties. In the present paper, we explore instead their connections to nondeterministic cellular automata, as well as their dynamical properties.

Avoshifts cover several previously-studied classes of subshifts, in particular group shifts, TEP shifts and TSSM shifts, which we discuss in Section~\ref{sec:Motivation}. We state our new contributions in Section~\ref{sec:Results}, mainly the representation of avoshifts in terms of nondeterministic cellular automata, and the density of periodic points and full shift factors obtained as corollaries (for unishifts). In Section~\ref{sec:Examples}, we show specifically how the Ledrappier subshift and the golden mean shift can be seen as spacetimes of nonuniform cellular automata.

\subsection{Motivation and background}
\label{sec:Motivation}

Let $A$ be a finite alphabet. A \emph{subshift} $X \subset A^{\Z^d}$ is a set that is topologically closed for the Cantor topology of $A^{\Z^d}$ and closed under the shift maps $\sigma_{\vec u}(x)_{\vec v} = x_{\vec v - \vec u}$. Its elements are called \emph{points} or \emph{configurations}. A \emph{subshift of finite type} or \emph{SFT} is a subshift defined by removing the configurations of $A^{\Z^d}$ whose shift-orbit intersects a particular clopen set $F$. A clopen set is finite union of cylinders $[P] = \{x \in A^G \;|\; x|D = P\}$ where $P : D \to A$ for finite $D \Subset \Z^d$, and the patterns $P$ corresponding to $F$ are called the \emph{forbidden patterns}.

It is well-known in the field of symbolic dynamics that one-dimensional SFTs (the case $d = 1$) are much more understandable than multidimensional SFTs ($d \geq 2$). For example, questions about multidimensional SFTs are usually undecidable, while for one-dimensional SFTs they are usually decidable (though there are exceptions in both cases). A recurring endeavor in the field of symbolic dynamics is to try to find classes of multidimensional subshifts which capture some of the interesting phenomena that appear in two-dimensional SFTs, but where at least some properties are decidable or otherwise understandable. 

Group shifts -- groups where the elements are points of the subshift, and group operations are continuous and shift-commuting -- are the most successful such a class.

\begin{example}
\label{ex:Ledrappier}
The \emph{Ledrappier subshift} $X$ is the set of two-dimensional binary configurations $x \in \Z_2^{\Z^2}$ satisfying
\[ \forall \vec v \in \Z^2: x_{\vec v} + x_{\vec v + e_1} + x_{\vec v + e_2} = 0_{\Z_2} \]
where $e_1 = (1,0), e_2 = (0,1)$. The set $\Z_2^{\Z^2}$ is a group under cellwise addition, and $X$ is a shift-invariant closed subgroup of it, i.e.\ a group shift. \qee
\end{example}

We list some of the key properties of group shifts:
\begin{enumerate}
\item\label{it:sft} They are SFTs. 
\cite{Ki87,KiSc88}
\item\label{it:lang} Their languages are decidable (uniformly). \cite{KiSc88}
\item\label{it:lang} Their inclusions are decidable. \cite{BeKa24}
\item\label{it:trace} Their lower-dimensional traces (meaning projections to lower dimension) can be computed. \cite{BeKa24}
\item\label{it:dpp} Their periodic points are dense. \cite{KiSc88}
\item\label{it:eef} They factor onto any full shift with lower or equal entropy. \cite{BoSc08}
\item\label{it:meas} They admit a natural measure. \cite{Ha33}
\end{enumerate}
``Uniformly'' means that the computation takes also the subshift as an input (in the form of forbidden patterns). In item~\ref{it:eef}, actually the subshifts always factor onto an equal entropy full shift. The natural measure referred to here is the Haar measure; we do not attempt to give a general definition of natural here.

For one-dimensional SFTs, not all these results hold, but we have a good understanding of their failure. In particular, we mention that for topologically transitive SFTs there is a natural measure called the Parry measure \cite{Pa66a} and periodic points are dense \cite{LiMa95}. As for item~\ref{it:eef}, if $X$ is a one-dimensional SFT and $h(X) \geq \log n$ then it factors onto the $n$-symbol full shift \cite[Theorem 5.5.8]{LiMa95} (furthermore, the factor map has a section given by a finite-state transducer).

In \cite{Sa22d}, the author introduced the so-called TEP subshifts (and more generally $k$-TEP subshifts). The Ledrappier example is also the most important example of a TEP subshift, but TEP subshifts generalize the idea that the defining rule ($x_{\vec v} + x_{\vec v + e_1} + x_{\vec v + e_2} = 0_{\Z_2}$) is ``permutive'' in the coordinates $(\vec v, \vec v + e_1, \vec v + e_2)$, while group shifts instead generalize the idea that the set of its solutions is a group. TEP subshifts share all of the properties listed above for group shifts, but they are essentially orthogonal to group shifts, indeed as shown in \cite[Example~5.14]{Sa22d}, even simple TEP subshifts (of algebraic origin) may not commute with any group structure (or even a unital magma or quasigroup operation).

The natural measure $\mu$ of a TEP subshift $X$, called the TEP measure, is defined by taking the marginal distribution on every finite convex subset of $\Z^d$ to be uniform over globally valid patterns:
\begin{equation}
\label{eq:unimeasure}
\forall C \Subset \Z^d\mbox{ convex}: \forall p \in X|C: \mu(p) = 1/\#X|C.
\end{equation}

Gluing/mixing properties can also have similar implications. 
We mention in particular the topological strong spatial mixing or \emph{TSSM} \cite{Br18}, and \emph{safe symbols} \cite{Br18}. 
We say $0 \in A$ is a safe symbol for an SFT $X$ if turning a non-$0$ symbol to $0$ in a valid configuration can never introduce a forbidden pattern; 
and $X$ has TSSM with gap $n \in \N$, if for any $U, V, S \subset \Z^d$ such that $\dist(U, V) \geq n$, and for every $u \in A^U$, $v \in A^V$ and $s \in A^S$, we have
\[ u \cup s \sqsubset X, v \sqcup s \sqsubset X \implies u \cup v \sqsubset s \sqsubset X, \]
where $\cup$ denotes union of patterns and $\sqsubset X$ denotes global validity in $X$. The TTSM property is implied by having a safe symbol. TSSM subshifts have all the properties listed above for group shifts, apart from having a (unique) natural measure. 

\begin{example}
\label{ex:GoldenMean}
A simple example of a subshift with a safe symbol is the two-dimensional golden mean shift $Y$, a.k.a.\ the hard-core model. It is the set of two-dimensional binary configurations $x \in \{0, 1\}^{\Z^2}$ satisfying
\[ \forall \vec v \in \Z^2: x_{\vec v} = 1 \implies x_{\vec v + e_1} = x_{\vec v + e_2} = 0. \]
Equivalently, two $1$s cannot appear consecutively in the horizontal or vertical direction. The symbol $0$ is a safe symbol, since changing a $1$ to a $0$ cannot result in a new adjacent pair of $1$s. \qee
\end{example}

The first three columns of Table~\ref{tab:Properties} summarize these properties of group shifts, TEP subshifts and TSSM shifts.

\begin{table}
\begin{tabular}{|l|l|l|l|l|l|l|}
\hline
			& $\Z$-SFT & group shift	& $k$-TEP		& TSSM				& unishift 		& avoshift \\
\hline
SFT			& yes & yes	\cite{KiSc88}& yes			& yes \cite{Br18}		& yes \cite{Sa24}		& yes \cite{Sa24} \\
\hline 
comp.\ language	& yes & yes	\cite{KiSc88}& yes 			& yes	\cite{Br18}		& yes \cite{Sa24}		& yes \cite{Sa24} \\
\hline
comp.\ inclusion & yes & yes	\cite{BeKa24}& yes			& yes				& yes	\cite{Sa24}		& yes \cite{Sa24} \\
\hline
comp.\ traces & yes & yes	\cite{BeKa24}& yes			& yes				& yes	\cite{Sa24}		& yes \cite{Sa24} \\
\hline 
natural measure	& ok & yes	\cite{Ha33}	& yes			& no \cite{Br18,BuSt94}	& \textbf{yes}	 		& no \cite{BrMcPa18,BuSt94}	 \\
\hline
periodic points	& ok & dense \cite{KiSc88}& dense		& dense \cite{Br18}		& \textbf{dense}		& \textbf{exist} \\
\hline
full shift factors	& $\leq h$ & equal $h$ \cite{BoSc08}	& equal $h$	& $< h$ \cite{BrMcPa18}& \textbf{equal $h$}	& ? \\
\hline
\end{tabular}
\caption{A summary of properties of various known examples of avoshifts (and the class of all avoshifts), highlighting the new contributions. ``Comp.'' means computable; $\leq h$ means a factor map exists to any full shift with less or equal entropy (and similarly for equal $h$ and $< h$); ``ok'' means that the issue is completely understood.}
\label{tab:Properties}
\end{table}

\subsection{New results}
\label{sec:Results}

\emph{Avoshifts} are \emph{subshifts} (shift-invariant closed subsets of $A^{\Z^d}$ for finite alphabet $A$) such that whenever $D$ and $D \cup \{\vec 0\}$ are convex subsets of $\Z^d$ (intersections of Euclidean convex sets with $\Z^d$), there exists a finite set $C \Subset D$ such that for all restrictions $x \in X|D$, the possible extensions of $x$ to $\vec 0$ (giving an element of $X|D \cup \{\vec 0\}$) are determined by the restriction $x|C$. Avoshifts generalize the group shifts,\footnote{This is true up to conjugacy -- the group operation must be defined cellwise for the avoshift property to hold, but all group shifts admit a conjugate form where this is true.} TEP subshifts and TSSM subshifts discussed in the previous section, as shown in \cite{Sa24}.

The following theorem is the main new result of the present paper for avoshifts. 
It was shown in \cite{Sa24} that in dimension $d = 1$, avoshifts are precisely the subshifts of finite type. Our theorem gives information on their structure in higher dimensions.

\begin{theorem}
\label{thm:AvoNS}
Let $X \subset A^{\Z^d}$ be an avoshift in dimension $d \geq 2$. Then there is an avoshift $Z \subset B^{\Z^d}$ obtained from $X$ by a change of basis of $\Z^d$ and applying a convex block coding, such that $Z$ is the two-sided spacetime subshift of a nondeterministic cellular automaton on a $(d-1)$-dimensional avoshift.
\end{theorem}

Convex block codings are a specific type of topological conjugacy where we record a convex-shaped context into each cell. A nondeterministic cellular automaton is a generalization of a classical cellular automaton, where there can be multiple (always at least one) choice for the image of a letter. A key point is that these choices are independent in different cells, which is the main improvement of the theorem over the structure theory in \cite{Sa24} (though the latter applies more generally to strongly polycyclic groups). By ``two-sided'' we mean that the possible extensions of a $(d-1)$-dimensional configuration are determined in both directions ($e_d$ and $-e_d$) by nondeterministic cellular automata.

The theorem seems to be ``new'' for all of the three examples we mentioned, i.e.\ group shifts, $k$-TEP and TSSM subshifts. For $k$-TEP, it is almost immediate (it follows from the fact that hyperplanes are convex sets and the main result of \cite{Sa22d}), and for TSSM it is straightforward to show (because TSSM are in fact \emph{uniformly} avo as shown in \cite{Sa24}). For group shifts, we do not know an essentially simpler proof.


In the case of $\Z^d$, all of the main results of \cite{Sa24} are immediate corollaries of this theorem. It also gives the following result which was not obtained in \cite{Sa24}:

\begin{theorem}
Every avoshift has a periodic point.
\end{theorem}

\emph{Unishifts} are the subshifts where whenever $D$ and $D \cup \{\vec 0\}$ are convex subsets of $\Z^d$, the number of extensions of a $D$-shaped pattern $p$ to $D \cup \{\vec v\}$ is independent of $p$. It was shown in \cite{Sa24} that unishifts are always avoshifts.

The proof of the existence of a natural measure for $k$-TEP subshifts in \cite{Sa22d} consists of showing that $k$-TEP subshifts have this defining property of unishifts, after which it is almost immediate that Equation~\ref{eq:unimeasure} defines a shift-invariant measure. Accordingly, the following theorem is almost immediate:

\begin{theorem}
Equation~\ref{eq:unimeasure} defines a measure of maximal entropy for unishifts.
\end{theorem}

For unishifts, we obtain the following two results, as rather immediate corollaries of Theorem~\ref{thm:AvoNS}.

\begin{theorem}
Every unishift has dense totally periodic points.
\end{theorem}

\begin{theorem}
Every unishift has an equal entropy full shift factor.
\end{theorem}

The theorems above are in fact proved for various classes of subshifts coming from nondeterministic cellular automata, and then concluded for avoshifts and unishifts from Theorem~\ref{thm:AvoNS}. Of course, these results may be of independent interest, although we do not list them here separately.

Compared to \cite{Sa24}, the approach in the present paper is more abstract, and we develop an abstract theory of ``manuals'' and ``building'', which abstract away the way we build sets from other sets with local rules in avoshifts, and give rise to a combinatorial theory which could have some independent interest (or other applications).

While this may introduce a psychological hurdle for the reader, it clarifies induction steps, since now stepping to lower dimension is much easier (it simply gives another manual), while in \cite{Sa24} we had to explicitly keep track of a higher-dimensional context. Crucially, it also helps clarify the new key technical idea of ``permissivity'', a type of commutativity that the local rules of avoshifts always satisfy.

\subsection{Nondeterministic cellular automata representation for the examples}
\label{sec:Examples}

We immediately explain what the nondeterministic cellular automata look like for Example~\ref{ex:Ledrappier} and Example~\ref{ex:GoldenMean}, as many readers are likely to be deeply familiar with these examples, and they may be sufficient for understanding what kind of structure the main theorem provides. The examples do not illustrate how the theory constructs the presentation. For basic definitions, the reader should look at Section~\ref{sec:Definitions}.

The precise definitions related to nondetermistic cellular automata are given in Section~\ref{sec:Nondeterministic}, but we give a brief explanation. The idea is that nondeterministic cellular automata are a variant of cellular automata where the local rule may output multiple symbols, and any choice of symbol is valid. Our definition is equivalent to that in \cite{LeMa14}, although we concentrate on the subshifts they define rather than their nondeterministic dynamicals.

More precisely, a configuration in a two-dimensional subshift can be thought of as a sequence of rows, and the nondeterministic cellular automats will take the contents $x$  of row $i$ (resp.\ row $i+1$) as input, and produce any possible contents $y$ on the $(i+1)$th row (resp.\ $i$th row) by applying its local rule to $x$. The rule is $y_j \in F(x|j+N)$ (for all $j \in \Z$) where $N \Subset \Z$ is a finite \emph{neighborhood} and $F : A^N \to \pow(A)$ is a \emph{local rule}. Crucially, $F(w)$ is always nonempty, the choices $y_j \in F(x|j+N)$ are independent for each different $j$, and the rule $F$ is also the same for each $j$ (in $F(x|j+N)$, we identify $x|j+N$ with an $N$-shaped pattern in an obvious way).

\subsubsection{The Ledrappier example}

We start with the Ledrappier example $X$ from Example~\ref{ex:Ledrappier}. For missing definitions, see Section~\ref{sec:Definitions}. Indexing rows from south to north on the second axis and columns from west to east, note that the dynamics of the Ledrappier example is north deterministic (row $i$ determines row $i+1$), to the east (column $i$ determines column $i+1$) and also southwest (diagonal $i+1$ determines diagonal $i$). However, the other direction is not given by a nondeterministic cellular automaton in any of these cases: a single arbitrary choice can first be made, and the rest of the row is determined (in a highly non-local fashion).

Consider now the basis $\bar e_1 = (1, 1), \bar e_2 = (0, 1)$ for $\Z^2$. Let $M$ be the matrix with $\bar e_1, \bar e_2$ as rows, and define $X'$ as the subshift with configurations $x' \in X'$ such that for some $x \in X$ we have $x'_{\vec v} = 1 \iff x_{\vec v M} = 1$, equivalently $x'_{\vec v M^{-1}} = 1 \iff x_{\vec v} = 1$ where $M^{-1} = \left(\begin{smallmatrix} 1 & -1 \\ 0 & 1 \end{smallmatrix}\right)$. The rule determining $x \in X'$ becomes
\[ \forall \vec v \in \Z^2: x_{\vec v} + x_{\vec v + (1, -1)} + x_{\vec v + (0, 1)} = 0. \]

The subshift is now north and south deterministic. Still, it is not described by nondeterministic cellular automata, because the forbidden patterns look at three consecutive rows (in cellular automata terminology, one might call this a ``second-order'' nondeterministic cellular automaton). So next, we block the subshift $X'$ by the convex set $N = \{(0,0), (0,1)\}$, taking the new content of a cell to be the contents of $N$, to get a new subshift $X''$ over the new alphabet of $X''$ is $B = \Z_2^2$, which for indexing purposes we split as $B^{\Z^2} = \Z_2^{\Z^2} \times \Z_2^{\Z^2}$. Now $(y, z) \in X''$ means that there exists $x \in X'$ such that for all $\vec v \in \Z^2$, $y_{\vec v} = x_{\vec v}$ and $z_{\vec v} = x_{\vec v + e_2}$. Now $(y, z) \in X''$ if and only if
\[ \forall \vec v: y_{\vec v + e_2} + z_{\vec v} \mbox{ and } z_{\vec v} + y_{\vec v + e_1} + z_{\vec v + e_2}. \]
The first condition here checks that $(y, z)$ comes from the $N$-blocking of some binary configuration, and the second checks the forbidden pattern at $\vec v + e_2$.

Now $X''$ is fully determined by deterministic cellular automata (which are special cases of nondeterministic cellular automata, where the local rule always gives a single output symbol): the $\Z$-trace (i.e.\ set of valid rows, a.k.a.\ projective subdynamics) of $X''$ is $(\Z_2^2)^\Z$; to go from row $i$ to row $i+1$ we can apply the cellular automaton rule $f((a, b), (c, d)) = a + d$ (with neighborhood $\{0,1\}$); and to go from row $i+1$ to row $i$ we can apply the cellular automaton rule $f((a, b), (c, d)) = a + b$ (with neighborhood $\{-1,0\}$).

\subsubsection{The golden mean shift}

Now consider the golden mean shift example $Y$ from Example~\ref{ex:GoldenMean}. From the contents of row $i$ we again cannot compute those of row $i+1$ by a nondeterministic local rule. For example, if row $i$ contains only $0$s, then on row $i+1$ we would like to be able to write any word from the golden mean shift in dimension $1$ (any word without consecutive $1$s). Since a nondeterministic cellular automaton must choose the symbols at each coordinate independently, and it sees only $0$s in its input it cannot ever write $1$ in this case, as nothing would prevent writing two $1$s consecutively.

As in the Ledrappier example, use the basis $\bar e_1 = (1, 1), \bar e_2 = (0, 1)$ for $\Z^2$ to get a subshift $Y'$ where the rule for $x \in Y'$ is
\[ \forall \vec v \in \Z^2: x_{\vec v} = 1 \implies x_{\vec v + (1, -1)} = x_{\vec v + (0, 1)} = 0. \]

The subshift $Y'$ is determined by nondeterministic cellular automata: its $\Z$-trace is the one-dimensional golden mean shift; to go from row $i$ to row $i+1$ we can use the nondeterministic cellular automaton rule
\[ F(a) = \begin{cases} \{0\} & \mbox{ if } a = 1 \\
\{0,1\} & \mbox{otherwise} \end{cases} \]
with neighborhood $\{0\}$; and to go from row $i+1$ to row $i$ we can use the same rule with neighborhood $\{-1\}$.

Note that the $\Z$-trace is not a full shift, but is an avoshift (on $\Z$ this is equivalent to being an SFT).

\section{Definitions and basic facts}
\label{sec:Definitions}

\subsection{Numbers}

Our convention is $\N = \{0,1,\ldots\}$. We write $\Z_n$ for the additive group of integers modulo $n$, and identify the elements with the representatives $\{0, 1, \ldots, n-1\}$.

\subsection{Sets}

If $B, C$ are sets, write $B \sqcup C$ for their disjoint union in the sense that $B \sqcup C = B \cup C$ and the notation suggests that the union is disjoint. We write $B \subset C$ for $B \subseteq C$. Write $B \Subset C$ for a finite subset. Write $\pow(B) = \{C \;|\; C \subset B\}$ for the powerset of $B$. Write the cardinality of a set $B$ as $\#B$ or $|B|$.

We say $C$ is \emph{(sandwiched) between} $A$ and $B$ if $A \subset B \subset C$. We also refer to $C$ as an \emph{intermediate} set. A simple but important observation is that if $(O, \leq)$ is an ordered set and $f : \pow(S) \to O$ is monotone in the sense that $A \subset B \implies f(A) \leq f(B)$ or $A \subset B \implies f(A) \geq f(B)$, then whenever $f(A) = f(B)$ we have $f(C) = f(A) = f(B)$ for all $C$ between $A$ and $B$.

Often, to check $A \subset B$ for two sets, it is helpful to \emph{check inclusion separately in $C$ and elsewhere}, by which we mean apply the following (trivial) lemma:

\begin{lemma}
\label{lem:SetInclusionCases}
Let $X$ be a set and $A, B \subset X$. If $C \subset X$, then
\begin{align*}
A \subset B &\iff (A \cap C) \subset (B \cap C) \mbox{ and }  (A \cap (X \setminus C)) \subset (B \cap (X \setminus C)) \\
&\iff (A \cap C) \subset B \mbox{ and } (A \cap (X \setminus C)) \subset B.
\end{align*}
Similar formulas hold for $A = B$.
\end{lemma}

\subsection{Symbolic dynamics}

An \emph{alphabet} is a finite set $A$. The \emph{full shift} (on alphabet $A$ and in dimension $d$) is the topological dynamical system with points $A^{\Z^d}$ and $\Z^d$ acting by \emph{shift maps} $\sigma_{\vec v}(x)_{\vec u} = x_{\vec v + \vec u}$. Here $A^{\Z^d}$ has the product topology, so the shift maps are homeomorphisms. A \emph{subshift} is a topologically closed shift-invariant set $X \subset A^{\Z^d}$.

A \emph{pattern} $p$ is an element of $A^D$ for some $D \subset \Z^d$ called its \emph{shape}. Write $\vec v + p$ for the pattern $q \in A^{\vec v + D}$ defined by $q_{\vec v + \vec u} = p_{\vec u}$. If $f : A^D \to B$ and $p \in A^E$, we may write $f(p) = f(q)$ where $q = \vec v + p$ if such $\vec v \in \Z^d$ exists (in which case it is unique). We write the restriction of $p$ to $E \subset D$ as $p|E$. For a subshift $X$, we write $X|D = \{x|D \;|\; x \in X\}$. Note that if $x \in X$ then $x$ can be seen as a pattern and $\vec v + x = \sigma_{-\vec v}(x)$.

If $p \in A^D$ is a pattern and $X$ a subshift, we write $p \sqsubset X$ if $X|D \ni p$ and say $p$ \emph{appears} in $X$. The patterns that appear in $X$ form its \emph{language}. If $p \in A^D, q \in A^E$, write $p \cup q \in A^{D \cup E}$ for the unique pattern $r$ such that $r|D = p, r|E = q$, if it exists. Write this as $p \sqcup q$ if $D \cap E = \emptyset$.
We define the \emph{follower set} of $p$ at $\vec v$ as
\[ \follow_X(p, E) = \{q \in A^E \;|\; p \cup q \cup X\} \]
We may write a vector $\vec v$ in place of $E$ to mean $\{\vec v\}$. IF $p \in A^D$, patterns $\follow_X(p, D \cup E)$ are also called \emph{extensions} of $p$ to $E$.

If $x \in A^{\Z^d}$, write $S_i(x)$ for the configuration $y \in A^{\Z^{d-1}}$ defined by $y_{\vec v} = x_{(\vec v, i)}$. The \emph{$(d-1)$}-trace of a subshift $X \subset A^{\Z^d}$ is $\{S_0(x) \;|\; x \in X\}$.

We need two simple lemmas for follower sets. The first is obvious and the second is proved by a short compactness argument.

\begin{lemma}
\label{lem:FollowMonotone}
If $p \in A^D$ and $D' \subset D$, then $\follow_X(p, E) \subset \follow_X(p|D', E)$.
\end{lemma}

\begin{lemma}
\label{lem:FollowCompact}
If $p \in A^D$ and $E$ is finite, then there exists $D' \Subset D$ such that $\follow_X(p, E) = \follow_X(p|D', E)$.
\end{lemma}

A \emph{subshift of finite type} or \emph{SFT} is $X \subset A^{\Z^d}$ such that for some finite \emph{window} $W \Subset \Z^d$ and set of \emph{forbidden patterns} $P \subset A^W$ we have
\[ X = \{x \in A^{\Z^d} \;|\; \forall p \in P : p \not\sqsubset x \}. \]
Note that while we allow infinite patterns, forbidden patterns are finite.

A \emph{measure} on a subshift $X$ means a Borel probability measure. Equivalently, $\mu$ is a measure on $X$ if for each pattern $p \sqsubset X$, a number $\mu(p) \in [0,1]$ is associated, and if $p \in A^D$ and $\vec v \notin D$, then if we let
\[ B = \{q \in A^{D \cup \{\vec v\}} \;|\; q|D = P, q \sqsubset X\}, \]
then $\sum_{q \in B} \mu(q) = \mu(p)$. The connection to the usual notion of a measure is that $[p] = \{x \in X \;|\; x|D = p\}$ for the \emph{cylinder} defined by $p$, we have $\mu(p) = \mu([p])$.

Two subshifts $X, Y \subset A^{\Z^d}$ are \emph{conjugate} if there is a homeomorphism $\phi : X \to Y$ that intertwines the actions in the sense that $\phi \circ \sigma_{\vec v} = \sigma_{\vec v} \circ \phi$ for all $\vec v \in \Z^d$.

Recall that the \emph{topological entropy} of a subshift $X \subset A^{\Z^d}$ is defined as
\[ h(X) = \lim_n \frac{\log \# X|[0, n-1]^d}{n^d}. \]
More generally, we say that sets $F_i \Subset \Z^d$ form a \emph{F\o{}lner sequence} if for all $S \Subset \Z^d$, we have $|F_i \; \Delta \; F_i+S|/|F_i| \rightarrow 0$ where $F_i + S = \{\vec u + \vec v \;|\; \vec u \in F_i, \vec v \in S\}$. For any F\o{}lner sequence, the topological entropy can be calculated as
\[ h(X) = \lim_i \frac{\log \# X|F_i}{\# F_i}. \]

If $\mu$ is a shift-invariant measure on $X$, then its \emph{measure-theoretic entropy} (there are many equivalent definitions, this one is from \cite{QuTr00}) is 
\[ h_\mu(X) = \lim_n \frac{\sum_{p \in X|[0,n-1]^d} -\mu(p) \log \mu(p)}{n^d}. \]
Recall also the variational principle:
\[ h(X) = \max_{\mu} h_\mu(X) \]
where $\mu$ runs over the shift-invariant measures.

\subsection{Graphs}

A \emph{graph} is always directed, i.e.\ is $(V, E)$ where $E \subset V^2$. Elements of $V$ are \emph{vertices} of $G$, and those of $E$ \emph{edges} of $G$. Specifically, $(u, v) \in E$ is an edge \emph{from $u$ to $v$}. The \emph{in-degree} of $v \in V$ is $\#\{u \in V \;|\; (u, v) \in E\}$ and symmetrically define the \emph{out-degree}. If $(V_1, E_1), (V_2, E_2), \ldots, (V_n, E_n)$ are directed graphs, their \emph{edge-disjoint union} exists if $E_i \cap E_j = \emptyset$ whenever $i \neq j$, and is simply $(\bigcup_i V_i, \bigcup_i E_i)$. A \emph{subgraph} of $(V, E)$ is $(V', E')$ with $V' \subset V$ and $E' \subset E$.

Two directed graphs $(V, E)$ and $(V', E')$ are \emph{isomorphic} if there is a bijection $\phi : V \to V'$ such that $(u, v) \in E \iff (\phi(u), \phi(v)) \in E'$. A directed graph is a \emph{cycle} if it is isomorphic to $(\Z_n, \{(i, i+1) \;|\; i \in \Z_n\})$.

A \emph{bipartite graph} is a graph of the form $(A \sqcup B, E)$ where $A, B \neq \emptyset$ and $E \subset A \times B$. We call $A$ the \emph{left vertices} and $B$ the \emph{right vertices} (the choice may not be unique, but can be thought of as part of the structure).


\begin{lemma}
\label{lem:UnionOfCycles}
Let $G$ be a finite directed graph where for some $k$, every vertex has in- and out-degree $k$. Then $G$ is an edge-disjoint union of cycles.
\end{lemma}

\begin{proof}
Consider the bipartite graph with left vertices $V_L$, which are just a copy of vertices $G$, and right vertices $V_R$, which are also a copy of vertices of $G$, and an edge from $u \in V_L$ to $v \in V_R$ for each edge $(u, v)$ of $G$.

From the degrees we have the condition of Hall's theorem that the open neighborhood of any $A$ has cardinality at least $|A|$. Therefore, there is a perfect matching, i.e.\ a set of edges with exactly one edge from each $x \in V_L$, and with distinct endpoints in $V_R$. Since $|V_L| = |V_R|$, there is also exactly one incoming edge for each $y \in V_R$. This shows that $G$ has a subset of edges giving constant in- and out-degree $1$ to each node. Following these edges shows that this subgraph is a disjoint union of cycles. Removing them from the original graph $G$, we can proceed by induction with the in- and out-degrees decreased by one.
\end{proof}

\subsection{Geometric conventions}

Write $S_k = \Z^{d-1} \times \{k\}$. The \emph{$k$th slice} of a set $E \subset \Z^d$ is $E \cap S_k$. Define $T_{i,k} = S_i \cup S_{i+1} \cup S_{i+2} \cup \cdots \cup S_k$, $T_i = T_{0,i}$ and $T_\infty = \bigcup_i T_i$.

If $E \subset \Z^d$, write $\vec u + E = \{\vec u + \vec v \;|\; \vec v \in E\}$, and similarly for $E + \vec v$, $\vec v - E$, $E - \vec v$ (in particular, $-$ does not denote removal of elements). Write $e_i$ for the $i$th canonical basis vector so $\Z^d = \langle \{e_i \;|\; i = 1, \ldots, n\} \rangle$ as an additive abelian group.

For $\vec v \in \Z^d$ with last coordinate zero, write $\pi\vec v$ for the vector with the last zero coordinate removed, i.e.\ $\pi : S_0 \to \Z^{d-1}$ (where $S_0 \subset \Z^d$). We index coordinates of vectors $\vec v \in \R^d$ by positive numbers, so $\vec v = (\vec v_1, \ldots, \vec v_d)$. For $\vec u \in \R^{d-1}$ and $a \in \R$, write $(a, \vec u)$ for the vector $\vec v \in \R^d$ with $\vec v_1 = a$ and $\vec v_{i+1} = \vec u_i$ for $i \geq 1$. We define $(\vec u, a)$ symmetrically.

We write $\dist(\vec u, \vec v)$ for the Euclidean metric on $\R^d$, and for $A, B \subset \Z^d$ we use the conventions $\dist(A, B) = \inf_{\vec u \in A, \vec v \in B} \dist(\vec u, \vec v)$ and $\dist(\vec v, B) = \dist(\{\vec v\}, B)$. If $\vec v \in \Z^d$, the \emph{ball} of radius $r$ is $B_r(\vec v) = \{\vec u \in \Z^d \;|\; \dist(\vec v, \vec u) \leq r\}$, and $B_r = B_r(\vec 0)$. Note that our balls are always discrete.

Define $\vec 0^d$ as the zero vector of $\R^d$, or just $\vec 0$ with the dimension inferred from context. More generally write $\vec a$ for the all-$a$ vector. For the purpose of matrix multiplication (but for no other purpose), we see vectors are column vectors. (Alternatively, the reader can think of our matrices $M$ as just linear maps instead.) The dot product of two vectors is $\vec u \cdot \vec v = \sum_i \vec u_i \vec v_i$.

\section{More geometry}

In this section, we define a number of geometric terms and study their basic properties. In particular, we introduce convex sets, inductive intervals, half-spaces, and the notion of an extension (meaning a pair of sets).

As a general convention, we mention that in what follows, typically an unqualified adjective/noun will refer to a subset of $\Z^d$, while the $\R^d$ versions are called \emph{real}.

\emph{Real intervals} are connected subsets of $\R$ that are not equal to $\R$, and \emph{intervals} are intersections of geometric intervals with $\Z$. Note that intervals can be finite, or half-infinite in either direction, but $\Z$ is not considered an interval.

A \emph{rational closed halfspace} is $\bar H_{\vec u, r} = \{\vec v \in \Z^d \;|\; \vec v \cdot \vec u \geq r\}$ for some $\vec u \in \Z^d \setminus \{\vec 0\}$ called the \emph{direction}, and it is required that $\gcd(\vec u_1, \ldots, \vec u_d) = 1$. This $\gcd = 1$ requirement on $\vec u$ ensures that the choice of $\vec u, r$ is canonical. A \emph{rational open halfspace} $H_{\vec v, r}^\circ$ is defined similarly, but using instead the formula $\vec v \cdot \vec u > r$.

In fact, rational closed halfspaces and rational open halfspaces are easily seen to be exactly the same sets, but with a different choice of $r$ for the same $\vec u$. This is because $H_{\vec v, r}^\circ = \bar H_{\vec v, r+1}$, which holds since we interpret the vectors in $\Z^d$, and dot products are integer-valued. Thus we also refer to rational closed halfspaces and rational open halfspaces also as just \emph{rational halfspaces} when the representation does not matter. A \emph{hyperplane} is a set of the form $\partial H_{\vec v, r} = \bar H_{\vec v, r} \setminus H_{\vec v, r}^\circ$.

A subset of $\R^d$ is \emph{real convex} if every straight line between points of the set stays in the set, and the \emph{real convex hull} of a set $A \subset \R^d$, denoted $\CHull(A)$, is the intersection of all real convex sets containing $A$. A subset of $\Z^d$ is \emph{convex} if it is the intersection of a real convex set with $\Z^d$. The following is immediate.

\begin{lemma}
A set $C \subset \Z^d$ is convex if and only if $C = \CHull(C) \cap \Z^d$.
\end{lemma}

A \emph{real polytope} in $\R^d$ is a compact set that is an intersection of finitely many closed halfspaces, or equivalently the convex hull of finitely many points \cite{Gr03}. A \emph{polytope} in $\Z^d$ is an intersection of a real polytope with $\Z^d$. The previous lemma shows that every finite convex set is a polytope.

The \emph{inductive intervals} are defined inductively. In $\Z$, the inductive intervals are the \emph{zero-grazing intervals}, meaning intervals (recall that for us, intervals can be half-infinite) such that $I \not\ni 0$, and $I \cup \{0\}$ is also an interval. In $\Z^d$, they are the sets of the form $(C \times \{\vec 0^{d-1}\}) \cup (\Z \times D)$ where $C$ is an inductive interval in $\Z$ and $D$ is an inductive interval in $\Z^{d-1}$. Note that $\emptyset$ is an inductive interval. Note also that inductive intervals are not a shift-invariant class of sets.

Another equivalent description, the \emph{interval list description}, of inductive intervals is that $E$ is an inductive interval if and only if for some zero-grazing intervals $I_1, \ldots, I_d$ in $\Z$, we have 
\[ E = \{\vec v \in \Z^d \;|\; (v_i \neq 0 \wedge \forall j > i: \vec v_j = 0) \implies \vec v_i \in I_i\}. \]

An \emph{extension} is a pair of sets $(E, E') \in \pow(\Z^d)^2$ where $E \subset E'$. If $\PP$ is a family of subsets of $\Z^d$, a \emph{$\PP$-extension} is an extension $(E, E')$ where both $E$ and $E'$ are in $\PP$ up to a shift. A $\PP$-extension $(E, E')$ is \emph{minimal} if there is no set $D \in \PP$ strictly between $E$ and $E'$. A \emph{singleton $\PP$-extension} is a $\PP$-extension $(E, E')$ where $|E' \setminus E| = 1$.
In particular, this gives meaning to the 12 terms \emph{(minimal/singleton) inductive interval extensions}, \emph{(minimal/singleton) convex extensions} and \emph{(minimal/singleton) rational halfspace extensions}.

\begin{lemma}
If $C$ is an inductive interval, then $C$ and $C \cup \{\vec 0\}$ are convex sets.
\end{lemma}

\begin{proof}
In dimension $1$, this is clear. Suppose $C \subset \Z, D \subset \Z^{d-1}$ and $A \cup B$ is an inductive interval, where $A = C \times \{\vec 0^{d-1}\}$ and $B = \Z \times D$. By induction, $D$ and $C$ are convex. It suffices to show that any real convex combination of finitely many elements from $A \cup B$ stays in $A \cup B$.

Since $A, B$ are convex, any such combination is a linear combination of an element of $\CHull(A) = \CHull(C) \times \{0\}^{d-1}$ and $\CHull(B) = \R \times \CHull(D)$. But clearly either the last $d-1$ coordinates give a vector in $D$, in which case, since the linear combination is an integer vector, the result is in $B$; or elements of $B$ are not used at all, and the result is in $A$. Thus, in both cases the result is in $A \times B$.
\end{proof}

The following lemma contains some observations about extensions of various types.

\begin{lemma}
\label{lem:ExtensionStuff}
\begin{itemize}
\item If $(E, E')$ is a minimal convex extension, then $|E' \setminus E| = 1$.
\item The minimal rational halfspace extensions are precisely ones of the form $E = H_{\vec v, r}^\circ$ and $E' = \bar H_{\vec v, r}$ for some $\vec v \in \Z^d$ and $r \in \Z$.
\item If $(E, E')$ is a singleton inductive interval extension with $E' \setminus E = \{\vec v\}$, then $-\vec v + E$ is an inductive interval, and either $- e_1 -\vec v + E'$ or $e_1 + -\vec v + E'$ is an inductive interval.
\end{itemize}
\end{lemma}

\begin{proof}
\textbf{Minimal convex extensions.} Suppose $(E, E')$ is a minimal convex extension. Let $a \in E' \setminus E$ be arbitrary. Then $\CHull(E \cup \{a\}) \supset E'$ since the extension is minimal, as otherwise we could take $D = \CHull(E \cup \{a\}) \cap \Z^d$ as an intermediate convex set. In particular, if $|E' \setminus E| > 1$, then $b \in (\CHull(E \cup \{a\}) \cap \Z^d) \setminus E$ for some $b \in E' \setminus E$. But it is easy to show that $\dist(b, E) < \dist(a, E)$ in this case. Thus we cannot have $a \notin \CHull(E \cup \{b\}) \cap \Z^d$, because a symmetric argument gives the contradiction $\dist(a, E) < \dist(b, E) < \dist(a, E)$. We conclude that $(E, E')$ must be a singleton extension.

\textbf{Minimal rational halfspace extensions.} Suppose then $(E, E')$ is a minimal rational halfspace extension. Observe that $H_{\vec v, r}^\circ = \bar H_{\vec v, r+1}$, so we may assume $E = H_{\vec v, r}^\circ$ and $E' = \bar H_{\vec u, r'}$ for some $\vec v, \vec u, r, r'$. Clearly $\vec u$ and $\vec v$ must be parallel, so in fact they must be equal (because of the $\gcd = 1$ requirement on the $\vec u_i$). Since the extension is minimal, clearly $r' = r$. It is also clear that $(H_{\vec v, r}^\circ, \bar H_{\vec v, r})$ is always a rational halfspace extension.

\textbf{inductive interval extensions.} Suppose $(E, E')$ is a singleton inductive interval extension. We start by observing that since $E'$ is the shift by some $\vec v = (a, \vec u)$ (where $a \in \Z, \vec u \in \Z^{d-1}$) of an inductive interval, we have $E' = A \sqcup B$ where $A = (a + C) \times \{\vec u\}$ where $C$ is a zero-grazing interval, and $B = \Z \times (\vec u + D)$ for $D$ an inductive interval in dimension $d-1$.

Observe that since $|E' \setminus E| = 1$, there is a vector in $E'$ that can be dropped so that the result is still a shift of an inductive interval. Clearly, this vector must be in $A$. Namely, if it were in $B$, then (because the union $A \cup B$ is disjoint) for some $\vec w \in \vec u + D$, $E$ would contain $((-\infty, b-1] \cup [b+1, \infty)) \times \{\vec w\}$ without containing $(b, \vec w)$, contradicting convexity of inductive intervals.

If the vector is in $A$, then $E = ((a + C') \times \{\vec u\}) \cup B$ where $C'$ is an interval, and $C = C' \cup \{n\}$ for some $n \in \Z$. Now let $\vec w = (-n - a, -\vec u)$ and replace $(E, E')$ by $(\vec w + E, \vec w + E')$ (noting that the original claim is invariant under shifting $E, E'$ simultaneously). Then $E' \setminus E = \{\vec 0\}$,
\[ E' = ((C-n) \times \{\vec 0\}) \cup (\Z \times D) \]
(note that $C-n$ is translation, not element removal) and
\[ E = ((C'-n) \times \{\vec 0\}) \cup (\Z \times D). \]
Since $D$ is an inductive interval and $C-n = (C'-n) \cup \{0\}$ is an interval, $C'-n$ is a zero-grazing interval, and thus $E$ is an inductive interval. If $C'-n \subset (0, \infty)$ then $E' + e_1$ is an inductive interval, and if $C'-n \subset (-\infty, 0)$ then $E' - e_1$ is an inductive interval.
\end{proof}

There are some subtleties here. It is \emph{not} true that to every convex set $C$ in $\Z^2$ you can add a single vector while remaining convex (but in this case, by the previous lemma there is no minimal convex extension of the form $(C, C')$). For finite convex sets, you can always find a singleton extension \cite{Sa22d}. In the case of inductive intervals, every inductive interval has a singleton extension, but one can find minimal inductive interval extensions which are not singleton extensions. (We omit the examples, as we do not strictly need these facts.)

%

The \emph{slantedness} of a vector $\vec u \in \Z^d$ is defined only for vectors with positive coordinates, and is the vector $(\vec u_2/\vec u_1, \vec u_3/\vec u_2, \ldots, \vec u_d/\vec u_{d-1})$. If $\vec v \in \R^{d-1}$, we say $\vec u$ is \emph{$\vec v$-slanted} if $\vec u_{i+1}/\vec u_i \geq \vec v_i$ for all $i$, i.e.\ its slantedness is coordinatewise greater than or equal to $\vec v$. A $\vec v$ used to give such a requirement is called a \emph{slant vector}. We often say a vector is \emph{sufficiently slanted}, to mean that it is $\vec v$-slanted for some finite set of vectors $\vec v$ that arise in a proof. The \emph{slantedness} of a hyperspace is the slantedness of its direction.

\begin{remark}
\label{rem:SlantednessJoin}
Note that any finite set of slantedness requirements $\vec v^1, \vec v^2, \ldots, \vec v^n$ can be joined into a single one by setting $\vec v_i = \max_j v^j_i$. Note also that a $\vec a$-slanted vector for $a > 1$ has $\vec u_j$ larger than $\vec u_i$ whenever $j > i$, so we may always assume that a sufficiently slanted vector $\vec v$ will be increasing when seen as a sequence of positive integers $(\vec v_1, \vec v_2, \ldots, \vec v_d)$. 
\end{remark}

\begin{lemma}
\label{lem:SlantednessOfSlice}
If $\vec u$ is sufficiently slanted, so is the vector $\vec u'$ obtained by removing the last coordinate. In particular, if a halfspace $H \subset \Z^d$ is sufficiently slanted, so is $\pi(S_0 \cap H)$.
\end{lemma}

More precisely, given any slant vector $\vec v' \in \R^{d-2}$ (describing lower bounds on $\vec u_{i+1}'/\vec u_i'$), we can find a slant vector $\vec v \in \R^{d-1}$ (describing sufficient lower bounds on $\vec u_{i+1}/\vec u_i$) to guarantee that $\vec u'$ is $\vec v'$-slanted. The proof of the first sentence is trivial ($\vec v$ is obtained from $\vec v'$ by adding any coordinate at the end), and for the second we observe $\pi(S_0 \cap H_{\vec v, r}) = H_{\vec v', r}$.


\begin{lemma}
\label{lem:BoundaryTranslation}
For all $\vec u \in \partial H_{\vec v, 0}$, the sets $H_{\vec v, r}^\circ$ and $\bar H_{\vec v, r}$ are closed under translation by $\vec u$ for all $r \in \Z$.
\end{lemma}

\begin{proof}
The assumption $\vec u \in \partial H_{\vec v, 0}$ means precisely $\vec u \cdot \vec v = 0$. Adding such a vector to $\vec w$ does not affect its dot product with $\vec v$.
\end{proof}

We now generalize halfspaces to slants where the last coordinate may be infinite, though we emphasize that ``sufficiently slanted'' will never refer to such generalized slants. If $\vec u \in \Z^{d-1} \times \{\pm\infty\}$, we define $\bar H_{\vec u, r} = \{\vec v \in \Z^d \;|\; \vec v \cdot \vec u \geq r\}$ with the same formula but with the usual extension of addition to infinite values. In other words, if $\vec u_d = (-1)^a \infty$ and $(-1)^b \vec v_d > 0$ then $\vec v \cdot \vec u = (-1)^{a+b} \infty$, and $-\infty < n < \infty$, and $n + \pm\infty = \pm\infty + n = \pm\infty - n = \pm\infty$ for all $n \in \Z$. Note that $\pm\infty$ can appear only once in the sum, so the sum is well-defined. We analogously define the open version $H_{\vec u, r}^\circ$.



\begin{lemma}
\label{lem:Inside}
For all $r$, if $\vec u$ is sufficiently slanted, then for all $\vec v \in \bar H_{\vec u, 0} \cap S_0$ we have $B_r(\vec v) \cap T_{0,\infty} \cap \bar H_{\vec u, 0} = B_r(\vec v) \cap T_{0,\infty} \cap \bar H_{\vec u', 0}$, where $\vec u'$ is obtained from $\vec u$ by replacing the last coordinate with $\infty$. The same holds for the open variants.
\end{lemma}

\begin{proof}
Suppose $\vec v \in \bar H_{\vec u, 0} \cap S_0$, i.e.\ $\vec v \in S_0$ and $\vec v \cdot \vec v \geq 0$.

We check the equality separately in $S_0$ and elsewhere (i.e.\ apply Lemma~\ref{lem:SetInclusionCases} to the slice $S_0$). In the slice, clearly
\[ S_0 \cap B_r \cap T_{0,\infty} \cap \bar H_{\vec u, 0} = S_0 \cap B_r \cap T_{0,\infty} \cap \bar H_{\vec u', 0}, \]
since the dot products do not involve the last coordinate.

Suppose then $\vec w \in S_0^c \cap B_r \cap T_{0,\infty}$, so $\vec w_d > 0$. Then obviously $\vec v + \vec w$ belongs to $\bar H_{\vec u', 0}$, since $\vec u_d' = \infty$ implies an infinite dot product for $(\vec v + \vec w) \cdot \vec u'$. In the other hand, if $\vec u$ is sufficiently slanted, then
\begin{align*}
(\vec v + \vec w) \cdot \vec u &= (\vec v \cdot \vec u) + (\vec w_d \cdot \vec u_d + \sum_{i < d} \vec w_i \cdot \vec u_i) \\
&\geq 0 + \vec u_d - m \sum_{i < d} |\vec u_i| \geq 0
\end{align*}
where $m = \max_{i < d} |\vec w_i|$. Thus, $\vec v + \vec w$ also belongs to $\bar H_{\vec u, 0}$.
\end{proof}




\begin{lemma}
\label{lem:HalfSpaceThing}
Let $H \subset \Z^d$ be a rational half-space whose direction is not $-e_d$. Then $H$ is a countable increasing union of shifts of $H \cap T_{i,\infty}$.
\end{lemma}

\begin{proof}
Let $H = \bar H_{\vec v, r}$. If $H$ has direction $\vec v = e_d$ then $H = T_{r, \infty}$ 
and the claim is clear. Otherwise, by the assumption the direction of $H$ is neither $e_d$ nor $-e_d$, so by the $\gcd = 1$ requirement on the direction we can find $\vec u \in \partial H_{\vec v, 0}$ with an arbitrarily large or small $d$-coordinate.

By Lemma~\ref{lem:BoundaryTranslation} we have for $\vec u \in \partial H_{\vec v, 0}$ the equation $\vec u + H = H$, and thus we have $\vec u + (H \cap T_{i,\infty}) = H \cap T_{i + \vec u_d,\infty} \subset H$, in particular taking $\vec u \in \partial H_{\vec v, 0}$ with $\vec u_d$ tending to $-\infty$, we have
\[ \bigcup_{\vec u \in \partial H_{\vec v, 0}} \vec u + (H \cap T_{i,\infty}) = H. \qedhere \] 
\end{proof}

\section{An abstract construction theory}

\subsection{Rules and manuals}


We describe an abstract construction theory that describes conditions under which new subsets of $\Z^d$ can be built from known subsets. The basic object is the \emph{(construction) rule}, namely a pair $(C, D)$ where $\vec 0 \not\ni C \Subset D \subset G$. Such a rule states that we can add $\vec v$ to a known set, if this set contains all of $\vec v + C$, and is contained in $\vec v + D$. We call $C$ the \emph{support} and $D$ the \emph{guard} of the rule.
 
Next, we define some terminology for sets of rules. A set $\C$ of construction rules is called a \emph{manual}. A manual $\C$ is \emph{monotone up} if 
\[ D = \bigcup_i D_i \wedge (\forall i: (C, D_i) \in \C) \implies (C, D) \in \C. \]
We say it is \emph{down} if
\[ (C, D) \in \C, C \subset D' \subset D \implies (C, D') \in \C. \]
A manual is \emph{finite} if it is finite as a set. The \emph{size} of a manual is its cardinality as a set. A \emph{submanual} is a manual contained in it as a set.

The terms ``down'' and ``monotone up'' refer to analogous properties of set systems. Namely, a family of sets $\C \subset \pow(\Z^d)$ is usually called \emph{down} if $C \in \C$ and $D \subset C$ imply $D \in \C$. On the other hand, a family of sets $\C \subset \pow(\Z^d)$ is sometimes called a \emph{monotone class} in measure theory if it is closed under countable increasing unions and countable decreasing intersections. We drop ``class'' and add ``up'' to clarify that intersections are not allowed. 

\begin{remark}
The intuition behind rules is the following. To each set $E \subset \Z^d$ is associated some kind of object $O(E)$ that we want to build (in our main application, $O(E) = X|E$ for a subshift $X$). A rule $(C, D)$ states that if we know $O(E)$, and $\vec v + C \subset E \subset \vec v + D$, then this particular rule can be used to build $O(E \cup \{\vec v\})$. Although not directly visible in the formalism, in our applications the rule determines a certain ``way'' of performing the extension, and $(C, D)$ is saying that the extension from $E$ to $E \cup \{\vec v\}$ can be done in some uniform way for any set $E$ sandwiched between $\vec v + C$ and $\vec v + D$.
\end{remark}

To (partially) formalize the intuition, if $\C$ is a manual, we say there is a \emph{$\C$-construction step}, or just \emph{$\C$-step}, from $A \subset G$ to $B \subset G$ if $B = A \sqcup \{\vec v\}$ and for some rule $R = (C, D) \in \C$, we have $\vec v + C \subset A \subset \vec v + D$. We write this $A \nearrow_{R, g} B$ or just $A \nearrow_R B$ (since $B \setminus A = \{\vec v\}$ determines $\vec v$), or $A \nearrow_\C B$ if the choice of rule is irrelevant or clear. We also say that we \emph{apply} the rule $(C, D)$ \emph{at} $\vec v$ to $A$ to produce $B$, and that a rule \emph{applies} if it can be applied. We express $\vec v + C \subset A$ in words by saying \emph{the support is valid} and $A \subset \vec v + D$ by saying \emph{the guard is valid}.

A \emph{$\C$-blueprint} is a sequence $((R_1, \vec v_1), (R_2, \vec v_2), \ldots, (R_n, \vec v_n))$ with $R_i = (C_i, D_i) \in \C$ and $\vec v_i \in \Z^d$. We say it \emph{builds} $B$ from $A$ if 
\[ A \nearrow_{R_1,\vec v_1} A_1 \cdots A_{n-1} \nearrow_{R_n,\vec v_n} B. \]
Of course, $B = A \cup \{\vec v_1, \ldots, \vec v_n\}$, and we say the blueprint is \emph{valid} for $A$, if it builds that set $B$ from $A$ (i.e.\ the rules apply at each successive step). In the context of a blueprint, the sets $\vec v_i + C_i, \vec v_i + D_i$ are called \emph{positioned supports} and \emph{positioned guards} respectively.

In this case we also say $B$ is \emph{directly $\C$-buildable} from $A$ if $A \nearrow_\C^{*} B$. If $R_i = (C_i, D_i)$, then we say more specifically say that $B$ is directly $\C$-buildable from $A$ \emph{with support} $E$ if $(\bigcup_i \vec v_i + C_i) \cap A \subset E \subset A$. In other words, $E$ should contain all positioned supports. 
In these terms and later ones, we may drop $\C$ if it is clear from context.

\begin{definition}
We say $B \supset A$ is \emph{$\C$-buildable} from $A$, or $\C$ \emph{builds} $B$ from $A$, and we write $A \rightarrow_\C B$ if for every finite subset $B' \Subset B$, we have $A \nearrow_{\C}^* A \cup B''$ for some $B' \subset B'' \Subset B$. 
\end{definition}

In other words, the definition says that we can directly build arbitrarily large finite extensions of $A$ which stay inside $B$.

The reason our definition of building only requires building finite subsets is that it makes building transitive. If we worked instead with direct building of infinite sets, then we would need to introduce blueprints of infinite length, and discuss their allowed order types, which would be an unnecessary complication for present purposes. 

\begin{lemma}
\label{lem:BuildabilityTransitive}
Buildability is transitive. 
\end{lemma}

\begin{proof}
Suppose $A \rightarrow_\C B \rightarrow_\C C$ (in particular $A \subset B \subset C$). Let $C' \Subset C$ be arbitrary. Since $B \rightarrow_\C C$, we have $B \nearrow_\C^* B \cup C''$ where $C' \subset C'' \Subset C$. Thus there is a blueprint $((R_1, \vec v_1), \ldots, (R_k, \vec v_k)$ that builds $B \cup C''$ from $B$ (and we can take $C'' = \{\vec v_1, \ldots, \vec v_k\}$). If $R_i = (D_i, E_i)$, and let $B' \Subset B$ be any finite subset containing all of the $\vec v_i + D_i$. Then the same blueprint is valid for $F$ for all $B' \subset F \Subset B$: the guards are valid because the sets are smaller at each rule application, and supports are valid because $F$ contains $B'$.

Thus for all $B' \subset F \Subset B$, we have $F \nearrow_\C^* F \cup C''$. Since $A \rightarrow_\C B$, we have $A \nearrow_\C^* A \cup B''$ for some $B''$ satisfying $B' \subset B'' \Subset B$. Letting $F = A \cup B''$ we have $A \nearrow_\C^* A \cup B'' = F \nearrow_\C^* F \cup C'' = A \cup B'' \cup C''$, thus $A \nearrow_\C^* A \cup G$ where $G = B'' \cup C''$ is a finite set containing $C'$.
\end{proof}

We will sometimes abuse terminology and say $B$ is buildable from $A$ even if $B \not\supset A$, in which case we mean $A \cup B$ is buildable from $A$. However, we warn the reader that Lemma~\ref{lem:BuildabilityTransitive} above is not true with this convention.

Now we define a key notion which has no parallel in \cite{Sa24}.

\begin{definition}
A manual is \emph{permissive} if 
\[ (C, D), (C', D') \in \C \mbox{ and } \vec v + C' \subset D \mbox{ and } D \cup \{\vec 0\} \subset \vec v + D' \implies (C, D \cup \{\vec v\}) \in \C. \]
\end{definition}

In words, the assumption states that if we can apply $(C', D')$ at $\vec v$ before or after applying $(C, D)$ at $\vec 0$, then we can add $\vec v$ to the guard $D$. 

\begin{remark}
Let us attempt to explain informally how one should think about this property. Recall from the intuition given above for rules that we are building some objects $O(E)$ on subsets $E$ of $\Z^d$, and for any set $E$ sandwiched between $\vec v + C$ and $\vec v + D$, the extension of $O(E)$ to $O(E \cup \{\vec v\})$ is somehow built in a uniform way. If the assumptions of permissivity hold, we can build $O(D \cup \{\vec 0\} \cup \{\vec v\})$ by first applying $(C, D)$ at $\vec 0$ and then $(C', D')$ at $\vec v$ (which is possible since $\vec v + C \subset D \cup \{\vec 0\} \subset \vec v + D'$).

We now observe that since the latter extension happens the same way no matter how we extended $E$ to $\vec v$, actually we could have applied $(C', D')$ at $\vec v$ first (which is possible since $\vec v + C' \subset D \subset \vec v + D'$), i.e.\ we could just as well add $\vec v$ to the guard of $(C, D)$. This informal idea will readily translate to the proofs of permissivity of the avo- and unimanuals in Lemma~\ref{lem:AvomanualPermissive}.
\end{remark}





We say that a manual $\C$ \emph{handles} a family of extensions $\mathcal{E}$ if for all extensions $(E, E') \in \mathcal{E}$ it builds $E'$ from $E$. We say that a manual $\C$ \emph{$k$-handles} such $\mathcal{E}$ if 
there is a submanual $\C' \subset \C$ of size $k$ that handles $\mathcal{E}$. 
We say $\C$ \emph{finitely handles} a family of extensions if it $k$-handles them for some finite $k$. 
We say $\C$ \emph{builds} a family of sets $\mathcal{H}$ if it handles the family of extensions $\{(\emptyset, H) \;|\; H \in \mathcal{H}\}$ (i.e.\ builds each $H \in \mathcal{H}$ from $\emptyset$), and we define the terminology \emph{$k$-builds} and \emph{finitely builds} as with handling. Note that ``handling extensions'' is just alternate terminology for building sets from other sets. Thus for example, the following is immediate from Lemma~\ref{lem:BuildabilityTransitive}:

\begin{lemma}
\label{lem:HandlingTransitive}
If a manual handles the extensions $(E, E')$ and $(E', E'')$ then it handles the extension $(E, E'')$.
\end{lemma}


We say that $\C$ \emph{handles in parallel the extension $(E, E')$} 
if there is a rule $(C, D) \in \C$ such that for all $\vec v \in E' \setminus E$, we have $\vec v + C \subset E$ and $E' \setminus \{\vec v\} \subset \vec v + D$. We note that while ``handles in parallel an extension'' is clumsy English, we apply the verb to objects with a long definition and splitting the verb would make some of our statements hard to parse.

\begin{remark}
As an important terminological clarification, when we say a manual $\C$ finitely handles (or $k$-handles) all $\PP$-extensions, we mean that for every $\PP$-extension $(E, E')$, $\C$ finitely handles (or $k$-handles) $(E, E')$. When we want to say it handles the family of all $\PP$-extensions, we say just that. A similar remark holds for building all sets of some type versus building the family of all sets of that type, and also to parallel handling.
\end{remark}


\subsection{Manuals and geometry}

If $\C$ is a manual, then its \emph{trace of width $k$} with $k \in \N \cup \{\infty\}$, written $T_k\C$, is the set containing $(\pi(C \cap S_0), \pi(D \cap S_0))$ whenever $(C, D) \in \C$, $C \Subset T_k, T_{1,k} \subset D$. Of course, it is a set of rules, thus a manual itself. The trace of infinite width is also called the \emph{halfspace trace}. A ``trace'' of a manual can refer to any of these manuals.

We show that traces preserve many properties of a manual.

\begin{lemma}
\label{lem:TraceHandles}
Let $\D$ be a trace of the manual $\C$. Suppose $\C$ is down. Then 
\begin{itemize}
\item $\D$ is down;
\item if $\C$ is monotone up, so is $\D$;
\item if $\C$ is permissive, so is $\D$; and
\item if $\C$ handles singleton inductive interval extensions, so does $\D$.
\end{itemize}
\end{lemma}


\begin{proof}
\textbf{Down.} We consider the $k$-trace (possibly $k = \infty$). Suppose $\C$ is down. If $(C, D) \in \D$, then $(C, D) = (\pi(\hat C \cap S_0), \pi(\hat D \cap S_0))$ where $(\hat C, \hat D) \in \C$, $\hat C \Subset T_k$, and $T_{1,k} \subset \hat D$. If $C \subset E \subset D$, then $\hat C \subset E \cup T_{1,k} \subset \hat D$, so $(\hat C, E \cup T_{1,k}) \in \C$ (by the down property). This implies 
\[ (\pi(\hat C \cap S_0), \pi((E \cup T_{1,k}) \cap S_0)) = (C, E) \in \D. \] 

\textbf{Monotonicity up.} Suppose $\C$ is monotone up (and is also down). Suppose $(C, D_i) \in \D$ where $D_i$ are an increasing sequence of sets. We have $(C_i, D_i) = (\pi(\hat C_i \cap S_0), \pi(\hat D_i \cap S_0))$ where $(\hat C, \hat D_i) \in \C$, $\hat C \Subset T_k$, and $T_{1,k} \subset \hat D$. Since $\C$ is down, we may assume $\hat D_i = \pi^{-1}(D) \sqcup T_{1,k}$ for all $i$. Then $\hat D_i$ is increasing, so $(\hat C, \hat D) \in \C$ for $\hat D = \bigcup_i \hat D_i$. Then 
\[ (\pi(\hat C \cap S_0), \pi(\hat D \cap S_0)) = (C, \bigcup_i D_i) = (C, D) \in \D. \]

 
\textbf{Permissivity.} Suppose $\C$ is permissive (and down). Suppose
\[ (C, D), (C', D') \in \D \mbox{ and } \vec u + C' \subset D \mbox{ and } D \cup \{\vec 0\} \subset \vec u + D' \]
for some $\vec u \in \Z^{d-1}$. We need to show $(C, D \cup \{\vec u\}) \in \D$. Let $\vec v = \pi^{-1} \vec u$. By the definition of $\D$, $(C, D) = (\pi(\hat C \cap S_0), \pi(\hat D \cap S_0))$ where $(\hat C, \hat D) \in \C$, $\hat C \Subset T_k$ and $T_{1,k} \subset \hat D$, and similarly $(C', D') = (\pi(\hat C' \cap S_0), \pi(\hat D' \cap S_0))$ where $(\hat C', \hat D') \in \C$, $\hat C' \Subset T_k$ and $T_{1,k} \subset \hat D'$.

Since $\C$ is down, we may assume $\hat D = \pi^{-1}(D) \sqcup T_{1,k}$ and $\hat D' = \pi^{-1}(D') \sqcup T_{1,k}$. Then $\vec u + C' \subset D$ and $D \cup \{\vec 0\} \subset \vec u + D'$ imply $\vec v + \hat C' \subset \hat D$ and $\hat D \cup \{\vec 0\} \subset \vec v + \hat D'$. From permissivity of $\C$ we obtain $(\hat C, \hat D \cup \{\vec v\}) \in \C$. 
Then
\[ (\pi(\hat C \cap S_0), \pi((\hat D \cup \{\vec v\}) \cap S_0)) = (C, D \cup \{\vec u\}) \in \D. \] 

\textbf{Handling singleton inductive interval extensions.} Suppose $\C$ handles singleton inductive interval extensions (and is down). Let $(E, E')$ be a singleton inductive interval extension in dimension $d-1$. By Lemma~\ref{lem:ExtensionStuff}, after a shift we may suppose that $E' \setminus E = \{\vec 0\}$, where $E$ is an inductive interval, and $E' + e_1$ is an inductive interval (the case $E' - e_1$ being symmetric). Set $\bar E = \pi^{-1}(E) \cup T_{1,k}, \bar E' = \pi^{-1}(E') \cup T_{1,k}$. It is clear that both are inductive intervals, so $(\bar E, \bar E')$ is a singleton inductive interval extension in dimension $d$. Thus $\C$ handles it, i.e.\ there is a rule $(\bar C, \bar D) \in \C$ such that $\bar C \Subset \bar E \subset \bar D$. In particular $\bar C \subset T_k$ and $T_{1,k} \subset \bar D$, so letting $C = \pi(\bar C \cap S_0), D = \pi(\bar D \cap S_0)$ we have $(C, D) \in \D$. This rule handles the extension $(E, E')$, because $C \subset E \subset D$ follows from $\bar C \Subset \bar E \subset \bar D$ by intersecting with $S_0$ and applying $\pi$.
\end{proof}

\begin{lemma}
\label{lem:BuildClosure}
For any manual, the set of buildable sets is closed under shifts and increasing unions.
\end{lemma}

\begin{proof}
Let $C$ be buildable. Then $\vec v + C$ is buildable, because if $D \Subset \vec v + C$ is arbitrary, then some finite set $E$ between $D - \vec v \Subset C$ and $C$ is built by some blueprint $((R_1, \vec v^1), (R_2, \vec v^2), \ldots, (R_k, \vec v^k))$. It is easy to see that then $((R_1, \vec v^1 + \vec v), (R_2, \vec v^2 + \vec v), \ldots, (R_k, \vec v^k+ \vec v))$ builds $\vec v + E$, which is between $D$ and $\vec v + C$.

Let $C = \bigcup_i C_i$ where $C_i \subset C_{i+1}$. If $D \Subset C$ is finite, then $D \subset C_i$ for some $i$. Then there is a finite set $E$ between $D$ and $C_i$ which is built by the manual. But $E$ is also a set between $D$ and $C$, so indeed we can build arbitrarly large finite subsets of $C$.
\end{proof}

Let $\C$ be a manual. Let $N \Subset \Z^d$. The $N$-blocking of $\C$ is the manual $\C'$ containing $(C', D')$ precisely when $C' \Subset D' \subset \Z^d \setminus \{\vec 0\}$ and $N$ is directly $\C$-buildable from $D' + N$ with support $C' + N$. 
A \emph{convex blocking} is an $N$-blocking with (finite) convex $N$.

\begin{lemma}
\label{lem:ConvexManuals}
If $\C$ handles all singleton convex extensions, then so do all its convex blockings.
\end{lemma}

\begin{proof}
Let $\C$ be a manual handling all singleton convex extensions, and let $\C'$ be its $N$-blocking for convex $N$. Suppose that $H$ and $H \cup \{\vec 0\}$ are convex. We need to show that $\vec 0$ is buildable from $H$ in $\C'$, or equivalently that there is a rule $R' = (C', D') \in \C'$ such that $H \nearrow_{R'} H \cup \{\vec 0\}$.

Thus we need to find $C' \Subset D' \subset \Z^d \setminus \{\vec 0\}$ such that $N$ is directly $\C$-buildable from $D' + N$ with support $C' + N$. This is equivalent to the fact that $(D' + N) \cup N = (D' \cup \{\vec 0\}) + N$ is buildable from $D' + N$, as we can take the support $C' \Subset D'$ to be any finite set such that $C' + N$ contains the union of the positioned supports of the building blueprint.

To see that $(D' \cup \{\vec 0\}) + N$ is buildable from $D' + N$, we observe that since $D' + N$ is a convex subset of the convex set $(D' \cup \{\vec 0\}) + N$ with finite set difference, there is a sequence of convex sets $D' + N = D_0, D_1, D_2, \ldots, D_k = (D' \cup \{\vec 0\}) + N$ such that $D_{i+1} = D_i \cup \{\vec u_i\}$ for some vectors $\vec u_i$. This is a standard fact about convex sets, see e.g.\ \cite[Lemma~3.10]{Sa22d}.

Now, since $\C$ handles singleton convex extensions, it handles all the extensions $(D_i, D_{i+1})$, and thus by Lemma~\ref{lem:HandlingTransitive} it handles $(D_0, D_k) = (D' + N, (D' \cup \{\vec 0\}) + N)$.
\end{proof}

\subsection{A theorem about manuals}

In this section, we prove our main result about manuals, from which all the subshift-specific results of this paper follow with relative ease. It is the following:

\begin{theorem}
\label{thm:MainTechnical}
If a manual is down, monotone up and permissive and handles all singleton inductive interval extensions, then it handles in parallel the family of all sufficiently slanted minimal rational halfspace extensions.
\end{theorem}

Let us unravel the statement a bit, to clarify some fine points. The requirement of handling all singleton inductive interval extensions means that whenever $E, E'$ are shifts of inductive intervals, and $|E' \setminus E| = 1$, then some rule can be used to build $E'$ from $E$. By Lemma~\ref{lem:ExtensionStuff}, this means precisely that to all inductive intervals $E$, we can add $\vec 0$ by using some rule. Handling in parallel the family of all sufficiently slanted minimal rational halfspace extensions means in turn (again by Lemma~\ref{lem:ExtensionStuff}) that there is a single rule $(C, D)$ and a slant vector $\vec v \in \R^{d-1}$ such that for all $\vec v$-slanted positive integer vectors $\vec w \in \Z^d$, $(C, D)$ handles in parallel the extension $(H_{\vec w, 0}^\circ, \bar H_{\vec w, 0})$. 

Before we prove this, we need a technical lemma.

\begin{lemma}
\label{lem:BuildHalfspaces}
If a manual is down and handles all singleton inductive interval extensions, then it finitely builds the family of all sufficiently slanted rational halfspaces.
\end{lemma}


We note that inductive intervals are in particular convex sets, so it suffices that the manual handles all singleton convex extensions.

\begin{proof}
We prove this by induction on dimension. It suffices to consider halfspaces $H = \bar H_{\vec w, 0}$, since we can translate any halfspace to this form. Let $\C$ be our manual. 
 
\textbf{Dimension one.} Suppose first that $d = 1$.\footnote{We could alternatively give appropriate definitions for everything in the case $d = 0$, so that the one-dimensional case would also be covered by the induction step (and the base of the induction would be trivial), but it seems much clearer to explain the one-dimensional case explicitly, as the general case is much longer, but follows the same idea.} Then $H = [0, \infty)$.

Since $\C$ handles singleton inductive interval extensions, it handles the extension $((0, \infty), [0, \infty))$ meaning that there is a rule of the form $R_\infty = (C, D)$ where $C \Subset (1, \infty) \subset D$ in $\C$, in particular $C \subset [1, r]$ for some $r$. Since $\C$ handles also singleton extensions of finite inductive intervals, it contains rules $R_t = (C, D)$ where $C \subset [1, t] \subset D$ for all $t \in \N$ (where $[1,0] = \emptyset$). We claim that the manual $\{ R_\infty, R_0, R_1, \ldots, R_{r-1} \}$ builds $[0, \infty)$.

For this, it suffices to show that it builds the set $[0, m]$ for arbitrary $m \geq r$ (as this eventually contains every finite subset of $[0, \infty)$). Indeed, we have
\begin{align*} \emptyset &\nearrow_{R_0,m} \{m\} \\
&\nearrow_{R_1,m-1} \{m-1, m\} \\
&\nearrow_{R_2,m-2} \{m-2, m-1, m\} \\
&\cdots \\
&\nearrow_{R_{r-1,m-r+1}} \{m-r+1, m-r+1, \ldots, m\} \\
&\nearrow_{R_\infty,m-r} \{m-r, m-r+1, \ldots, m\} \\
&\cdots \\
&\nearrow_{R_\infty,0} \{0, 1, \ldots, m\}.
\end{align*}

\textbf{Dimensions greater than one.} Next, consider $d > 1$. Let $\C' = T_\infty \C$ be the halfspace trace of $\C$. By Lemma~\ref{lem:TraceHandles}, it is down and handles all singleton inductive interval extensions. Thus by induction there is a finite submanual $\mathcal{R}' = \{R_1, \ldots, R_k\}$ of $\C'$ that builds the family of all sufficiently slanted rational halfspaces. By the definition of $\C'$, each $R_i = (C', D')$ comes from a rule $\hat R_i = (C, D) \in \C$ where $C \subset T_{\infty}$ and $T_{1,\infty} \subset D$, by $C' = \pi(C \cap S_0), D' = \pi(D \cap S_0)$. 
We observe that the sets $C$ are finite subsets of $T_{\infty}$, and there are $k$ many of them, so for some $r$ we have $C \subset T_r$ for all $i$. Let $\hat{\mathcal{R}}' = \{\hat{R}_1, \ldots, \hat{R}_k\}$.

Next, $\C_t = T_t \C$ for $t \in [0, r-1]$. By Lemma~\ref{lem:TraceHandles}, it is down and handles all singleton inductive interval extensions. Again, there is a finite submanual $\mathcal{R}_t = \{R_{t,1}, \ldots, R_{t,k_t}\}$ of $\C_t$ that already builds the family of all sufficiently slanted rational halfspaces. By the definition of $\C_t$, each $R_{t,i} = (C', D')$ comes from a rule $\hat R_{t,i} = (C, D) \in \C$ where $C \subset T_t$, $T_{1,t} \subset D$, by $C' = \pi(C \cap S_0), D' = \pi(D \cap S_0)$. 



Now entirely analogously to the case $d = 1$, we will show that if $\vec w$ is sufficiently slanted, then
\begin{itemize}
\item $H \cap S_m$ is buildable using $\hat{\mathcal{R}}_0$,
\item the extension $(H \cap S_m, H \cap T_{m-1, m})$ is buildable using $\hat{\mathcal{R}}_1$,
\item $\ldots$
\item the extension $(H \cap T_{m-r+2, m}), H \cap T_{m-r+1, m})$ is buildable using $\hat{\mathcal{R}}_{r-1}$,
\item the extension $(H \cap T_{m-r+1, m}), H \cap T_{m-r, m})$ is buildable using $\hat{\mathcal{R}}'$,
\item $\ldots$
\item the extension $(H \cap T_{2,m}, H \cap T_{1,m})$ is buildable using $\hat{\mathcal{R}}'$,
\item the extension $(H \cap T_{1,m}, H \cap T_m)$ is buildable using $\hat{\mathcal{R}}'$.
\end{itemize}

We now write a formal proof of this argument, which is entirely straightforward but lengthy. To be able to cover in one go both the cases using $\hat{\mathcal{R}}_i$ and the cases using $\hat{\mathcal{R}}'$, we set $\hat{\mathcal{R}}_t = \hat{\mathcal{R}}'$ for $t \geq r$. We also write $\hat{\mathcal{R}}_t = \{R_{t,1}, \ldots, R_{t,k_t}\}$ in this case, with $k_t = k$ and $R_{t, i} = R_i$. 


So let us build $H \cap T_{m-t, m}$ from $H \cap T_{m-t+1, m}$ using $\hat{\mathcal{R}}_t$ for $t = 0, \ldots, m$. In the case $t = 0$, note that we want to build $H \cap T_{m, m} = H \cap S_m$ from $H \cap T_{m+1, m} = H \cap \emptyset = \emptyset$, so this case indeed corresponds to the first item. In the case $t = m$, we build $H \cap T_{0, m}$ from $H \cap T_{1, m}$ so this indeed corresponds to the last item.

We observe that by shifting by a vector on $\partial H$ with last coordinate equal to $-m + t$ (as in the proof of Lemma~\ref{lem:HalfSpaceThing}), we can equivalently build $H \cap T_{0, t}$ from $H \cap T_{1, t}$. 
By the definition of building, it suffices to build arbitrarily large finite subsets of the frontier
\[ H' = S_0 \cap H = H \cap T_{0, t} \setminus H \cap T_{1, t} \]
of the extension. In other words, it suffices to consider a finite subset $S \Subset H'$ and show that we can build a set between $S$ and $H'$ using $\hat{\mathcal{R}}_t$.

Now observe that $\pi H'$ is a halfspace in dimension $d-1$. By Lemma~\ref{lem:SlantednessOfSlice}, the slantedness of $\pi H'$ increases as a function of slantedness of $H$, thus for $H$ sufficiently slanted, also $\pi H'$ is sufficiently slanted so that $\mathcal{R}_t = \{R_{t,1}, \ldots, R_{t,k_t}\}$ builds arbitrarily large finite subsets of $\pi H'$. In particular if $S' = \pi(S)$, by the assumption, $\mathcal{R}_t$ builds some subset of $\pi H'$ larger than $S'$, i.e.\ there exists a valid blueprint $((b_1, \vec u_1), (b_2, \vec u_2), \ldots, (b_\ell, \vec u_\ell))$ for $\mathcal{R}_t$ such that
\[ S' \subset U = \{\vec u_i \;|\; i = 1, \ldots, \ell\} \subset \pi H'. \]

We will now define a blueprint $((\hat b_1, \vec v_1), (\hat b_2, \vec v_2), \ldots, (\hat b_\ell, \vec v_\ell))$ for $\hat{\mathcal{R}}_t$ such that if $H$ is sufficiently slanted, then this blueprint is valid and builds a set between $S$ and $H'$ from $H \cap T_{1, t}$ using $\hat{\mathcal{R}}_t$. For this, for $i = 1, \ldots, \ell$, define $\hat b_i = \hat R_{t, j}$ when $b_i = R_{t, j}$; and define $\vec v_i = \pi^{-1}(\vec u_i)$. First, we observe that the set built by this blueprint is $V = \{\vec v_i \;|\; i = 1, \ldots, \ell\}$, where $\vec v_i = \pi^{-1}(\vec u_i)$. Since $\pi^{-1}$ is a bijection taking $S'$ to $S$, $\pi H'$ to $H'$, and $U$ to $V$, we have $S \subset V \subset H'$, so our new blueprint indeed builds a set between $S$ and $H'$.

Now all that remains is to show that this blueprint is valid. Let $b_i = R_{t, j} = (C', D')$ and $\hat b_i = \hat R_{t, j} = (C, D)$. Recall that $C \subset T_t$, $T_{1,t} \subset D$, $C' = \pi(C \cap S_0)$, and $D' = \pi(D \cap S_0)$. Note that if $t < r$, then $C \subset T_t$, $T_{1,t} \subset D$ is directly in the assumptions, while if $t \geq r$, then the rule comes from $\mathcal{R}'$ and we have $C \subset T_r \subset T_t$ and $T_{1,t} \subset T_{1,\infty} \subset D$

\textbf{Checking the supports.} First we check the supports of the blueprint. Thus we need to show that
\begin{equation}
\label{eq:Inclusion}
\vec v_i + C \subset \{\vec v_s \;|\; s = 1, \ldots, i-1\} \cup (H \cap T_{1, t}).
\end{equation}
Since the $b_j$ form a valid blueprint that builds the set $U$ (from nothing), we know that their supports are valid, i.e.\
\begin{equation}
\label{eq:A1} \vec u_i + C' \subset \{\vec u_s \;|\; s = 1, \ldots, i-1\}.
\end{equation}

We check the inclusion~\eqref{eq:Inclusion} in the $0$-slice and elsewhere (in the sense of Lemma~\ref{lem:SetInclusionCases}). It is clear that in the $0$-slice, we have the desired containment, as in this slice the inclusion simplifies to
\[ \vec v_i + (C \cap S_0) \subset \{\vec v_s \;|\; s = 1, \ldots, i-1\} \]
which is immediate from applying the bijection $\pi^{-1}$ to \eqref{eq:A1}.

Since $C \subset T_t$ and $\vec v \in H'$, the remaining vectors outside the $S_0$-slice are of the form $\vec v \in \vec v_i + (C \cap T_{1, t})$. Since there are only finitely many rules $\hat R_{t, j}$ to consider, $\vec v$ is at a bounded distance from $\vec v$, thus at bounded distance from $H' = S_0 \cap H$. Since they belong to $T_{1, t} \subset T_{1, \infty}$, by Lemma~\ref{lem:Inside}, if $H$ is sufficiently slanted, we have $\vec v \in H$. And of course we then even have $\vec v \in H \cap T_{1, t}$, which completes the proof of the inclusion.

\textbf{Checking the guards.} Next, we check the guards of the blueprint. We need to show that
\[ \{\vec v_s \;|\; s = 1, \ldots, i-1\} \cup (H \cap T_{1,t}) \subset \vec v_i + D. \]
Since the $b_j$ form a valid blueprint that builds the set $U$ (from nothing), we know that their guards are valid, i.e.\
\begin{equation}
\label{eq:A2} \{\vec u_s \;|\; s = 1, \ldots, i-1\} \subset \vec u_i + D'.
\end{equation}

We apply Lemma~\ref{lem:SetInclusionCases} and check the inclusion separately in the $0$-slice and elsewhere. It is clear that in the $0$-slice, we have the desired containment, as in this slice the inclusion simplifies to
\[ \{\vec v_s \;|\; s = 1, \ldots, i-1\} \subset \vec v_i + (D \cap S_0). \]
which is immediate from applying the bijection $\pi^{-1}$ to \eqref{eq:A2}.

The remaining vectors are in $H \cap T_{1,t}$. We have by the assumption on $D$ (and since $\vec v_i \in S_0$) that $-\vec v_i + T_{1,t} = T_{1,t} \subset D$, from which
\[ H \cap T_{1,t} \subset T_{1,r} \subset \vec v_i + D, \]
which completes the proof of the inclusion.

\textbf{Wrapping up.} By Lemma~\ref{lem:BuildabilityTransitive}, combining the steps gives us that $H \cap T_m$ is buildable for all $m$ if $H$ is sufficiently slanted. Looking at the proof, we see that this sufficient slant only depends on geometric properties of the finitely many rules in $\hat{\mathcal{R}}$. Indeed, the choice of these rules implied some new slantedness requirements, but at this point $m$ was not yet fixed.

For each $m$, during the building process, we again required various slants, but in each case these only depended on the maximal norm of a vector in one of the sets $C$ appearing in the guards of the rules, given by Lemma~\ref{lem:Inside}. Thus, these requirements did not actually depend on $m$. All in all, we have only finitely many slantedness requirements on $H$, that allow all of the $H \cap T_m$ to be built, and we can combine into one using Remark~\ref{rem:SlantednessJoin}.

Then by Lemma~\ref{lem:BuildClosure} we can build $H \cap T_\infty$, and by Lemma~\ref{lem:HalfSpaceThing} and Lemma~\ref{lem:BuildClosure} we can then build $H$. All of this is done using the finite manual $\hat{\mathcal{R}}$, so $H$ is indeed finitely built.
\end{proof}


Now we prove Theorem~\ref{thm:MainTechnical}.

\begin{proof}[Proof of Theorem~\ref{thm:MainTechnical}]
Let $\C$ be a down, monotone up and permissive manual that handles all inductive interval extensions. We prove by induction on dimension that $\C$ handles in parallel the family of all sufficiently slanted rational halfspace extensions. 


The case $d = 1$ is trivial: $\C$ handles the extension $((0, \infty), [0, \infty))$ so some rule handles it. But up to translation, $((0, \infty), [0, \infty))$ is the only rational halfspace extension. Since it is a singleton extension, the rule handles it in parallel. 

Now suppose $d \geq 2$, and suppose $\vec w \in \Z^d$ is sufficiently slanted (we will state the precise slantedness requirements as they arise). We will consider the halfspace extension $(H_{0, \vec w}^\circ, \bar H_{0, \vec w})$. By Lemma~\ref{lem:TraceHandles}, $T_\infty \C$ is a down, monotone up and permissive manual that handles all singleton inductive interval extensions in dimension $d-1$. By induction, $T_\infty \C$ handles in parallel the family of all sufficiently slanted minimal rational halfspace extensions in dimension $d-1$. Let $R' = (C', D')$ be the rule of $T_\infty \C$ that handles this family.

Let $\vec w'$ be the vector obtained from $\vec w$ by removing the last coordinate. By Lemma~\ref{lem:SlantednessOfSlice}, if $\vec w$ is sufficiently slanted, then $R'$ handles the extension $(H_{\vec w', 0}^\circ, \bar H_{\vec w', 0})$ in parallel. 
 Recall that this implies that $C' \subset H_{\vec w', 0}^\circ$ and $\bar H_{\vec w', 0} \setminus \{\vec 0\} \subset D'$. Since $T_\infty \C$ is down (by Lemma~\ref{lem:TraceHandles}), we may assume $D' = \bar H_{\vec w', 0} \setminus \{\vec 0\}$.


By the definition of $T_\infty \C$ 
there is a rule $R = (C, D) \in \C$ such that $C \subset T_\infty$, $T_{1,\infty} \subset D$, and $C' = \pi(C \cap S_0), D' = \pi(D \cap S_0)$. Since $\C$ is down, we can take $D = T_{1, \infty} \sqcup (D \cap S_0)$, and since we assumed $D' = \bar H_{\vec w', 0} \setminus \{\vec 0\}$, we have precisely
\[ D = T_{1, \infty} \cup (S_0 \cap (\bar H_{\vec w, 0} \setminus \{\vec 0\})). \]

We now \textit{attempt to} show that $R$ handles in parallel all sufficiently slanted parallel halfspace extensions. For this, we would need to show that if $H$ is sufficiently slanted, then we have $C \subset H_{\vec w, 0}^\circ$ (support) and $\bar H_{\vec w, 0} \setminus \{\vec 0\} \subset D$ (guard).

\textbf{Support and guard on the slice.} We use Lemma~\ref{lem:SetInclusionCases} with the slice $S_0$. Inside the slice, we have to check $C \cap S_0 \subset H_{\vec w, 0}^\circ$ and $(\bar H_{\vec w, 0} \cap S_0) \setminus \{\vec 0\} \subset D$. For this, simply apply $\pi^{-1}$ on both sides of the inclusions $C' \subset H_{\vec w', 0}^\circ$ and $\bar H_{\vec w', 0} \setminus \{\vec 0\} \subset D'$.

\textbf{Support.} Outside the slice, consider first $C \subset H_{\vec w, 0}^\circ$. Since $C \subset T_\infty$, it suffices to show that $C \cap T_{1, \infty} \subset H_{\vec w, 0}^\circ$. If $\vec w$ is sufficiently slanted, this is immediate from Lemma~\ref{lem:Inside}, since $C$ is finite.

\textbf{Guard.} As for $\bar H_{\vec w, 0} \setminus \{\vec 0\} \subset D$, use again Lemma~\ref{lem:SetInclusionCases} now on $T_{1, \infty}$ to get the cases
\[ \bar H_{\vec w, 0} \cap T_{1, \infty} \setminus \{\vec 0\} \subset D\]
and
\[ \bar H_{\vec w, 0} \cap T_{-\infty, -1} \setminus \{\vec 0\} \subset D. \]
The first of these is trivial since $D \supset T_{1, \infty}$.

Unfortunately,
\[ \bar H_{\vec w, 0} \cap T_{-\infty, -1} \setminus \{\vec 0\} \subset D \]
is not necessarily true, as we only know $T_{1,\infty} \subset D$. We will now use permissivity to show that there exists $(C, \bar D) \in \C$ where $\bar D$ contains $D \cup \bar H_{\vec w, 0} \cap T_{-\infty, -1}$. Then $(C, \bar D)$ still works for all the previous cases (since $C$ stays the same and $\bar D \supset D$), and also this last containment checks out.


We first apply Lemma~\ref{lem:BuildHalfspaces} to $T_\infty \C$ to obtain that there is a finite set of rules $(R_1, \ldots, R_k)$ in $T_\infty \C$ which can build every sufficiently slanted rational halfspace. By the definition of $T_\infty \C$, each $R_i = (A', B')$ comes from $\hat R_i = (A, B) \in \C$ such that $A' = \pi(A \cap S_0)$, $B' = \pi(B \cap S_0)$, $A \Subset T_\infty$ and $T_{1, \infty} \subset B$.



Let $D_i = D \cup (H_{\vec w, 0} \cap T_{-i, -1})$ and $\bar D = \bigcup_i D_i$, so indeed $\bar D = D \cup (H_{\vec w, 0} \cap T_{-\infty, -1})$. We note that $D_{i+1} \setminus D_i = H_{\vec w, 0} \cap S_{-i-1}$. This is because we chose $D = T_{1, \infty} \cup (D \cap S_0)$. Note also that
\[ D_i \supset \bar H_{\vec w, 0} \cap T_{-i, \infty} \setminus \{\vec 0\} \]
because $D \supset S_0 \cap (\bar H_{\vec w, 0} \setminus \{\vec 0\}))$ from above.

We prove by induction that $(C, D_i) \in \C$ for all $i$. We have $(C, D_0) = (C, D) \in \C$ by assumption. Assume $(C, D_i) \in \C$. For obtaining $(C, D_{i+1}) \in \C$, we will use that $\C$ is monotone up, so it suffices to show $(C, D_i \cup E) \in \C$ for arbitrarily large sets $E \subset D_{i+1} \setminus D_i = H_{\vec w, 0} \cap S_{-i-1}$. Fix a finite set $E \subset H_{\vec w, 0} \cap S_{-i-1}$.

Note that $H_{\vec w, 0} \cap S_{-i-1}$ contains precisely vectors $\vec u$ such that $\vec u_d = -i-1$ and $\vec u' \in H_{\vec w', (i+1) \vec w_d}$, where $\vec u'$ is obtained from $\vec u$ by removing the last coordinate. Let $E' \subset H_{\vec w', (i+1) \vec w_d}$ be obtained from $E$ by removing the last coordinates of each vector (which are all $-i-1$).

Since $(R_1, \ldots, R_k)$ can build every sufficiently slanted rational halfspace (in dimension $d-1$), it can build the halfspace $H_{\vec w', (i+1) \vec w_d}$, in particular it can directly build a finite set $F'$ with $E' \subset F' \Subset H_{\vec w', (i+1) \vec w_d}$. Let $((S_1', \vec v_1'), \ldots, (S_m', \vec v_m'))$ be a blueprint that builds such $F'$, where $S_\ell' \in \{R_j \;|\; j = 1, \ldots, k\}$ for all $\ell$. So $F' = \{\vec v_j' \;|\; j = 1, \ldots, m\}$, and if $R_j = (A', B')$ then
\[ \vec v_j' + A' \subset \{\vec v_\ell' \;|\; \ell < j\} \subset \vec v_j' + B'. \]

Define $\vec v_\ell = (\vec v_\ell', -i-1)$ and $S_\ell = \hat R_j \iff S_\ell' = R_j$. Of course, then $((S_1, \vec v_1), \ldots, (S_k, \vec v_k))$ builds a set $F$, which is obtained from $F'$ by adding $-i-1$ as the last coordinate of each vector. Such a set $F$ then contains $E$ and is contained in $H_{\vec w, 0} \cap S_{-i-1}$ so it suffices to show that this blueprint is valid for $D_i$.

For this, consider the application of $\hat R_j$ at $\vec v_j$. If $R_i = (A', B')$, then $\hat R_i = (A, B) \in \C$ such that $A' = \pi(A \cap S_0)$, $B' = \pi(B \cap S_0)$, $A \Subset T_\infty$ and $T_{1, \infty} \subset B$.

We need to check
\[ \vec v_j + A \subset D_i \cup \{\vec v_\ell \;|\; \ell < j\} \subset \vec v_j + B. \]
Apply Lemma~\ref{lem:SetInclusionCases} to the slice $S_{-i-1}$. In this slice, the formula simplifies to 
\[ \vec v_j + (A \cap S_{i-1}) \subset \{\vec v_\ell \;|\; \ell < j\} \subset \vec v_j + (B \cap S_{i-1}) \]
which is immediate from
\[ \vec v_j' + A' \subset \{\vec v_\ell' \;|\; \ell < j\} \subset \vec v_j' + B' \]
(which is the same chain of inclusions but with $-i-1$ as the last coordinate of each vector).

Outside the slice, for supports we check $\vec t \in \vec v_j + A$ where $\vec t_d > (\vec v_j)_d$. Since $A$ is one of finitely many sets, we have $\vec t \in B_r(\vec v_j)$ for some bounded $r$, and it follows from Lemma~\ref{lem:Inside} that if $\vec w$ is sufficiently slanted, then $\vec t \subset D_i$ because $D_i \supset \bar H_{\vec w, 0} \cap T_{-i, \infty}$. 

For guards outside the slice, we need to check $D_i \cup \{\vec v_\ell \;|\; \ell < j\} \subset \vec v_j' + \bar B$. But in vectors of $D_i \cup \{\vec v_\ell \;|\; \ell < j\}$ outside $S_0$, the $d$-coordinate is at least $-i$, and since $B$ contains $T_{1, \infty}$, $v_j' + B$ contains all such vectors.
\end{proof}

\section{Avoshifts, unishifts, and their manuals}

\begin{definition}
The \emph{avomanual} of a subshift $X$ is the manual containing $(C, D)$ if for all $x \in X$, $\follow_X(x|D, \vec 0) = \follow_X(x|C, \vec 0)$. A subshift $X$ is an \emph{avoshift}, or just \emph{avo}, \emph{for the family of sets $\SSS$}, if its avomanual handles all singleton $\SSS$-extensions.
\end{definition}

\begin{definition}
The subshift $X$ is an \emph{unishift}, or just \emph{uni}, for the family of sets $\SSS$, if whenever $D, D \cup \{\vec 0\} \in \Z^d + \SSS$, we have $|\follow_X(x|D, \vec 0)| = \follow_X(y|D, \vec 0)|$ for all $x, y \in X$.
\end{definition}

Here, $\Z^d + \SSS$ is the family of translations of sets $S \in \SSS$ (this is needed because our definition of inductive intervals is not shift-invariant). We can also state the definition of a unishift in terms of manuals. This is a little technical, but has the benefit that it allows direct application of manual theory to unishifts.

\begin{definition}
The \emph{unimanual} of a subshift $X$ is the manual containing $(C, D)$ if for all $x \in X$, $\follow_X(x|D, \vec 0) = \follow_X(x|C, \vec 0)$, and furthermore $|\follow_X(x|D, \vec 0)| = |\follow_X(y|D, \vec 0)|$ for all $x, y \in X$.
\end{definition}

Note that the unimanual is a subset of the avomanual.

For the next lemma, a family of sets $\C \subset \pow(\Z^d)$ is \emph{monotone up} if $C_i \in \C$ and $C_i \subset C_{i+1}$ for all $i$ imply $\bigcup_i C_i \in \C$. A family of sets $\C \subset \pow(\Z^d)$ has a \emph{basis of finite sets at the origin} if every $C \in \C$ such that $C \cup \{\vec 0\} \in \C$ is an increasing union of finite sets $D \Subset C$ such that $D, D \cup \{\vec 0\} \in \C$.

\begin{lemma}
Let $\SSS$ be a shift-invariant monotone up family of sets which has a basis of finite sets at the origin. A subshift $X$ is uni for $\SSS$ if and only if its unimanual handles all singleton $\SSS$-extensions.
\end{lemma}


\begin{proof}
Suppose $X$ is uni for the family $\SSS$. Consider a singleton $\SSS$-extension $(D, D \cup \{\vec 0\})$. Since $X$ is uni for $\SSS$, we have $|\follow_X(x|D, \vec 0)| = \follow_X(y|D, \vec 0)|$ for all $x, y \in X$. Let $z$ be arbitrary. Let $C_i \in \SSS$ be finite sets such that $D = \bigcup_i C_i$ and $C_i \cup \{\vec 0\} \in \SSS$. Then $|\follow_X(z|C_i, \vec 0)|$ is a decreasing sequence by Lemma~\ref{lem:FollowMonotone}, therefore $|\follow_X(z|C, \vec 0)| = |\follow_X(z|D, \vec 0)|$ for some $C = C_i$.

Since $C, C \cup \{\vec 0\} \in \SSS$, we have for any $x \in X$ that
\[ |\follow_X(x|C, \vec 0)| = |\follow_X(z|C, \vec 0)| = |\follow_X(z|D, \vec 0)| = |\follow_X(x|D, \vec 0)|, \]
which again by Lemma~\ref{lem:FollowMonotone} implies that $\follow_X(x|D, \vec 0) = \follow_X(x|C, \vec 0)$ for all $x \in X$, and furthermore $|\follow_X(x|D, \vec 0)| = |\follow_X(z|D, \vec 0)| = |\follow_X(y|D, \vec 0)|$ for all $x, y \in X$. We conclude that $(C, D)$ is in the unimanual. It clearly handles the extension $(D, D \cup \{\vec 0\})$.

In the other direction, suppose the unimanual handles all $\SSS$-extensions. Then in particular for any $\SSS$-extension $(D, D \cup \{\vec 0\})$ we have some rule $(C, D')$ in the unimanual such that $C \subset D, D \subset D'  \not\ni \vec 0$, and $\follow_X(x|D', \vec 0) = \follow_X(x|C, \vec 0)$ for all $x \in X$, and $|\follow_X(x|D', \vec 0)| = |\follow_X(y|D', \vec 0)|$ for all $y \in X$. Since $C \subset D \subset D'$, it follows from $\follow_X(x|D', \vec 0) = \follow_X(x|C, \vec 0)$ that $\follow_X(x|D, \vec 0) = \follow_X(x|D', \vec 0)$ for all $x \in X$, thus $|\follow_X(x|D, \vec 0)| = |\follow_X(y|D, \vec 0)|$ meaning $X$ is uni for $\SSS$.
\end{proof}

If $N \Subset \Z^d$ and $X \subset A^{\Z^d}$ is a subshift, the \emph{$N$-blocking} of $X$ is the subshift with alphabet $B = X|N$ and configurations
\[ Y = \{y \in B^{\Z^d} \;|\; \exists x \in X: x|\vec v + N = y_{\vec v}\} \]
where we identify patterns up to a shift. Note that $X$ and $Y$ are topologically conjugate (when $N$ is nonempty).

\begin{lemma}
Let $N \Subset \Z^d$ and let $X \subset A^{\Z^d}$ be a subshift with avomanual (resp.\ unimanual) $\C$. Them the avomanual (resp.\ unimanual) $\D$ of the $N$-blocking $Y$ of $X$ contains the $N$-blocking of $\C$.
\end{lemma}

\begin{proof}
Let $N, \C, \D, X, Y$ be as in the statement, and let $\phi : X \to Y$ be the $N$-blocking conjugacy. We show the claim for the avomanual, the proof for the unimanual being similar. Let $(C', D')$ be in the $N$-blocking of $\C$, meaning $C' \Subset D' \subset \Z^d \setminus \{\vec 0\}$ and $N$ is directly $\C$-buildable from $D' + N$ with support $C' + N$. This means there is a blueprint $(((C_1, D_1), \vec v_1), \ldots, ((C_k, D_k), \vec v_k))$ that builds it with positioned supports contained in $C' + N$. This in turn means $(C_i, D_i) \in \C$, (possibly up to removing unnecessary rule applications) $V_k = N$ where we write $V_j = \{\vec v_i \;|\; i = 1, \ldots, j\}$,
\[ \vec v_i + C_i \subset C' + N \cup V_{i-1}, \]
and
\[ D' + N \cup V_{i-1} \subset \vec v_i + D_i \]

We need to show $(C', D') \in \D$, i.e.\ that for all $y \in Y$, we have $\follow_X(y|D', \vec 0) = \follow_X(y|C', \vec 0)$. Let $D = D' + N$ and $C = C' + N$. Then in fact $\follow_Y(y|D', \vec 0) = \follow_X(\phi^{-1}(y)|D, N)$ and $\follow_Y(y|C', \vec 0) = \follow_X(\phi^{-1}(y)|C, N)$, with the understanding that on the right we have an $N$-pattern, and on the left we have a single symbol which represents an $N$-pattern.

Thus it suffices to show $\follow_X(x|C, N) = \follow_X(x|D, N)$ for all $x \in X$, and by Lemma~\ref{lem:FollowMonotone} it suffices to show $\follow_X(x|C, N) \subset \follow_X(x|D, N)$. For this, consider $p \in X|C \cup N$ (so $p|N \in \follow_X(p|C, N)$). Since $p|C \in X|C$, there exists $x \in X|D$ such that $x|C = p|C$. Define 
\[ x^i = x \sqcup (\vec v_1 \mapsto p_{\vec v_1}) \sqcup \cdots \sqcup (\vec v_i \mapsto p_{\vec v_i}). \]
In other words, we have $x^0 = x$, and $x^{i+1} = x^i \sqcup (\vec v_{i+1} \mapsto p_{\vec v_{i+1}})$.

We prove by induction that $x^i \sqsubset X$. Of course $x^0 = x \sqsubset X$. Now assume that $x^i \sqsubset X$. We now recall that $(C_{i+1}, D_{i+1}) \in \C$ i.e.\
\begin{equation}
\label{eq:Derpo}\follow_X(x|\vec v_{i+1} + C_{i+1}, \vec v_{i+1}) = \follow_X(x|\vec v_{i+1} + D_{i+1}, \vec v_{i+1})
\end{equation}
Since
\[ \vec v_{i+1} + C_{i+1} \subset (C' + N) \cup V_{i-1} \subset (D' + N) \cup V_i \subset \vec v_{i+1} + D_{i+1}, \]
it follows from Lemma~\ref{lem:FollowMonotone} that
\[ \follow_X(x^i|\vec v_{i+1} + C_{i+1}, \vec v_{i+1}) = \follow_X(x^i|(D' + N) \cup V_i, \vec v_{i+1}), \]
where we observed that $(D' + N) \cup V_i$ is precisely the domain of $x^i$, and applied \eqref{eq:Derpo} to any globally valid extension of $x^i$.

We claim we have $p_{\vec v_{i+1}} \in \follow_X(x^i|\vec v_{i+1} + C_{i+1}, \vec v_{i+1})$. This is because
\[ \vec v_{i+1} + C_{i+1} \subset C' + N \cup V_i \subset C \cup N, \]
which is the domain of $p$, and
\[ \follow_X(x^i|\vec v_{i+1} + C_{i+1}, \vec v_{i+1}) = \follow_X(p|\vec v_{i+1} + C_{i+1}, \vec v_{i+1}) \ni p_{\vec v_{i+1}} \]
since $p \sqsubset X$.

It follows that $p_{\vec v_{i+1}} \in \follow_X(x^i|(D' + N) \cup V_i, \vec v_{i+1})$, so $x^{i+1} = x^i \sqcup (\vec v_{i+1} \mapsto p_{\vec v_{i+1}})$ is valid.
\end{proof}

\begin{lemma}
Suppose $X$ is avo (resp.\ uni) for convex sets, and $Y$ is the $N$-blocking of $X$ for a convex finite set $N$. Then $Y$ is avo (resp.\ uni) for convex sets.
\end{lemma}

\begin{proof}
Let $\C$ be the avomanual (resp.\ unimanual) of $X$. Since $X$ is avo (resp.\ uni) the assumption is that $\C$ handles all singleton convex extensions. By Lemma~\ref{lem:ConvexManuals}, so do all its convex blockings. By the previous lemma, if $Y$ is the $N$-blocking of $X$, then its avomanual (resp.\ unimanual) $\D$ contains the $N$-blocking $\mathcal{E}$ of $\C$. If $\mathcal{E}$ handles all singleton convex extensions, $\D$ does as well.
\end{proof}

\begin{lemma}
\label{lem:AvomanualPermissive}
The avomanual (resp.\ unimanual) of any subshift is down, monotone up and permissive.
\end{lemma}

\begin{proof}
We show these for the avomanual, the proofs for the unimanual being similar. Let $\C$ be the avomanual of a subshift $X$.

\textbf{Down.} Let $(C, D) \in \C$, so that for all $x \in X$, $\follow_X(x|D, \vec 0) = \follow_X(x|C, \vec 0)$. If $C \subset D' \subset D$, then by Lemma~\ref{lem:FollowMonotone}
\[ \follow_X(x|D, \vec 0) \subset \follow_X(x|D', \vec 0) \subset \follow_X(x|C, \vec 0), \]
so $(C, D') \in \C$ as well.

\textbf{Monotonicity up.} Let $(C, D_i) \in \C$ and $D_i \subset D_{i+1}$ for all $i$. Then for all $x \in X$, $\follow_X(x|D_i, \vec 0) = \follow_X(x|C, \vec 0)$. Let $D = \bigcup_i D_i$, and observe that $\follow_X(x|D, \vec 0) \subset \follow_X(x|C, \vec 0)$ by Lemma~\ref{lem:FollowMonotone}. By Lemma~\ref{lem:FollowCompact}, $\follow_X(x|D, \vec 0) = \follow_X(x|E, \vec 0)$ for $E \Subset D$ large enough, in particular $E \subset D_i$ for some $i$, so
\[ \follow_X(x|D, \vec 0) = \follow_X(x|E, \vec 0) \supset \follow_X(x|D_i, \vec 0) = \follow_X(x|C, \vec 0). \]

\textbf{Permissivity.} Suppose $(C, D), (C', D') \in \C$,  $\vec v + C' \subset D$ and $D \cup \{\vec 0\} \subset \vec v + D'$. We need to show $(C, D \cup \{\vec v\}) \in \C$. Suppose not, so that there $x \in X|D$ such that
\[ b \in \follow_X(x, \vec 0) \setminus \follow_X(x \sqcup (\vec v \mapsto a), \vec 0) \]
for some $a, b \in A$ such that $y = x \sqcup (\vec v \mapsto a) \sqcup X$. We have also $z = x \sqcup (\vec 0 \mapsto b) \sqcup X$, since $b \in \follow_X(x, \vec 0)$.

Since
\[ \vec v + C' \subset D \subset D \cup \{\vec 0\} \subset \vec v + D', \]
we have $\follow_X(z, \vec v) = \follow_X(x, \vec v)$. In particular, $a \in \follow_X(z, \vec v)$ so $z \sqcup (\vec v \mapsto a) \sqcup X$. But
\[ z \sqcup (\vec v \mapsto a) = y \sqcup (\vec 0 \mapsto b) \]
so $b \in \follow_X(y, \vec 0)$, a contradiction.
\end{proof}

\section{Rotating and blocking an avoshift}

\begin{definition}
If $M \in \GL(d, \Z)$ and $x \in A^{\Z^d}$, define $y = Mx$ by $y_{\vec v} = x_{M^{-1}\vec v}$. If $X \subset A^{\Z^d}$, write $MX \subset A^{\Z^d}$ for $\{Mx \;|\; x \in X\}$.
\end{definition}

\begin{lemma}
\label{lem:MatrixAvo}
If $M \in \GL(d, \Z)$ and $X \subset A^{\Z^d}$, then $Y = MX$ is a subshift. The avomanual of $Y$ consists of pairs $(MC, MD)$ such that $(C, D)$ is in the avomanual of $X$.
\end{lemma}

\begin{proof}
Obviously $Y$ is closed. For shift-invariance, we have
\begin{align*}
\sigma_{\vec v}(Mx)_{\vec u} &= Mx_{\vec v + \vec u} \\
&= x_{M^{-1}(\vec v + \vec u)} \\
&= x_{M^{-1}(\vec v) + M^{-1}(\vec u)} \\
&= \sigma_{M^{-1}(\vec v)}(x)_{M^{-1}(\vec u)} \\
&= M(\sigma_{M^{-1}(\vec v)}(x)){\vec u}.
\end{align*}

For the second claim, it is easy to check that $X \mapsto MX$ is a group action of $\GL(d,\Z)$ on subshifts. Thus, it suffices to show that $(MC, MD)$ is in the avomanual of $Y$ whenever $(C, D)$ is in the avomanual of $X$. For this, we need to show that if $y \in Y$, then $\follow_X(y|D', \vec 0) = \follow_X(y|C', \vec 0)$. Say $y = Mx$. Since $y_{M \vec v} = x_{\vec v}$, $y|MD$ and $y|MC$ contain respectively the contents of $x|D$ and $x|C$. The latter determine the same followers at $\vec 0$ for $x|C, x|D$, and thus we have the same followers at $M \vec 0 = \vec 0$ for $y|MC, y|MD$.
\end{proof}

\section{Nondeterministic spacetimes}
\label{sec:Nondeterministic}

\begin{definition}
Let $Z \subset A^{\Z^{d-1}}$ be a nonempty subshift. Let $N \Subset \Z^{d-1}$ and $F : A^N \to \pow(A)$ satisfies $F(p) \neq \emptyset$ for all $p \in A^N$. Suppose that the relation
\[ R = \{ (z, z') \;|\; z \in Z, z' \in A^{\Z^{d-1}}, \forall \vec v \in \Z^{d-1}: z'_{\vec v} \in F(z|{\vec v + N}) \}. \]
is contained in $Z \times Z$. Then $R$ is the \emph{nondeterministic cellular automaton} on $Z$ defined by $F$. The \emph{nondeterministic spacetime} or \emph{NS} of $R$ is 
\[ \{x \in A^{\Z^d} \;|\; \mbox{for all } i \in \Z: S_i(x) \in Z \mbox{ and } (S_i(x), S_{i+1}(x)) \in R \}. \]
We say a subshift is a \emph{binondeterministic spacetime} or \emph{BS} if it is a nondeterministic spacetime, and its image under $M$ is also a nondeterministic spacetime, where $M$ reflects the $d$-dimensional basis vector. In dimension $1$, we take binondeterministic spacetime to mean a vertex shift (i.e.\ an SFT with window $\{0,1\}$).
\end{definition}

Let us make two preliminary remarks about this definition. First, if $F : A^N \to \pow(A)$ always gives a singleton set, then the NS subshift is the usual spacetime subshift defined e.g.\ in \cite{Sa14}, whose configurations are the spacetime diagrams of the cellular automaton. We drop the word ``subshift'' for short, following \cite{CyFrKr19}. Under the same restriction on both $F$ and the local rule for the inverted direction, BS subshifts correspond to spacetime diagrams of reversible cellular automata (shift-commuting self-homeomorphisms of the subshift).

Second, in general, it is clear that if $X \subset A^{\Z^d}$ is BS, then $X|\Z^{d-1} \times \{0\} = Z$ (identifying $\Z^{d-1} \times \{0\}$ with $\Z^d$ in the obvious way): Any $z \in Z$ can be positioned in $\Z^{d-1} \times \{0\}$ and then extended to a point of $X$ by using the local rules arbitrarily to determine the contents of slices above and below $z$. Because of the assumption $R \subset Z \times Z$, indeed inductively on $|i|$ we see that all slices $\Z^{d-1} \times \{i\}$ contain points of $Z$ (again up to identifying $\Z^{d-1} \times \{i\}$ with $\Z^{d-1}$), therefore the resulting point is indeed in $X$.

\begin{definition}
We say a $\Z^d$-subshift is a \emph{uniform nondetermistic spacetime} if it is a nondeterministic spacetime, and we can pick the rule $F : A^N \to \pow(A)$ so that all $F(p)$ have the same cardinality. A \emph{uniform binondetermistic spacetime} is a binondeterministic spacetime where this holds in both directions. In dimension $1$, we take a binondeterministic spacetime to mean a vertex shift where each node has the same number of valid successors, and the same number of valid predecessors.
\end{definition}

\begin{lemma}
\label{lem:Entropy}
The entropy of a uniform nondeterministic spacetime is $\log n$ where $n$ is the cardinality of any of the images $F(p)$.
\end{lemma}


\begin{proof}
Let $Z \subset A^{\Z^{d-1}}$, $N \Subset \Z^{d-1}$ and $F : A^N \to \pow(A)$ be as in the definition. First we observe that the spacetime is nonempty: pick any $z \in Z$ and apply $F$ arbitrarily to obtain a sequence $z_{-k} = z, z_{-k+1}, \ldots, z_k$ such that $(z_i, z_{i+1}) \in R$ for all $i$. We can see these as a pattern on $\Z^{d-1} \times [-k, k]$ and any limit point is in the INS.

Now suppose $N \subset B_r$. Let 
\[ F_k = \{(\vec u, i) \;|\; i \in [0, k], \vec u \in [-r(k - i), r(k - i)]^{d-1} \} \]
this is a ``spacetime cone'' for $F$ in the sense that if we know the bottom pattern in $[-rk,rk]^{d-1} \times \{0\}$, the possible contents for the rest of the pattern are given by the rule $F$. Clearly for this F\o{}lner sequence we have $\log(\# X|F_n)/\#F_n \rightarrow \log n$: We first choose the bottom pattern, noting that there is at least one choice by the previous paragraph, and there are too few choices to affect the entropy. Then, we give the choices made by $F$, and we have roughly $n^{F_k}$ of them (since $F_k$ are a F\o{}lner sequence and the bottom slice indeed takes up only a small proportion).
\end{proof}

It is clear from Lemma~\ref{lem:Entropy} that for a uniform binondeterministic spacetime, the constant cardinality of images $F(p)$ is the same for both directions

\begin{definition}
We say a $\Z^d$-subshift is an \emph{inductive binondeterministic spacetime} if it is a binondeterministic spacetime, and if $d \geq 2$, then its $\Z^{d-1}$-trace is an inductive binondeterministic spacetime. We similarly define \emph{inductive uniform binondeterministic spacetimes} or \emph{IUBS}.
\end{definition}

\begin{lemma}
\label{lem:IBSPeriodic}
Every INS (thus every IBS) has a totally periodic point.
\end{lemma}

\begin{proof}
This is clear in dimension $1$, where INS just means one-step SFT. We proceed by induction. The $\Z^{d-1}$-trace of an INS has a periodic point $x$. Now pick any section of $F : A^N \to \pow(A)$, i.e.\ a map $f : A^N \to A$ such that $f(p) \in F(p)$ for all $p$. Then $f$ is a cellular automaton, and the one-sided spacetime diagram starting $x$ gives an infinite sequence of periodic points with the same period. From this, we can extract a periodic point in the $d$-dimensional subshift.
\end{proof}

\begin{lemma}
\label{lem:IUBSPeriodic}
Every IUBS has dense totally periodic points.
\end{lemma}

\begin{proof}
Let $X$ be our subshift. Let $F, F'$ be the nondeterministic CA rules in the two directions.

In dimension $1$, UIBS means a vertex shift where each node has the same number of successors and predecessors, say $k$. By Lemma~\ref{lem:UnionOfCycles}, as a graph this shift is a union of edge-disjoint cycles, from which it is clear that periodic points are dense.

We proceed by induction on dimension. Consider any pattern $p$, which we may extend to a hypercubic pattern. Let $q$ be the bottom slice of $p$ (i.e.\ the elements with minimal $d$-coordinate). We may assume its support is in $\Z^{d-1} \times \{0\}$. Now the nondeterministic CA rule $F$ can be used to build $p$ from some $(\Z^{d-1} \times\{0\})$-extension of $q$. Let $Y$ be the trace in dimension $d-1$. By induction, $q$ appears at the origin in a totally periodic point with a period lattice $U$ which is of finite index in $\Z^{d-1}$, and where $U$ is very sparse. We may think of $U$ as also a subgroup of $\Z^d$ by appending $0$ in the vectors.

Now consider the subshift $Z$ of $X$ consisting of points that are $U$-periodic. These points are obtained precisely by taking $U$-periodic points in $Y$, and then applying specializations of $F$ and $F'$ which apply the same realization at $\vec v$ and $\vec v'$ whenever $\vec v \equiv \vec v' \bmod U$. If $U$ is sparse enough, then $p$ appears in $Z$.

Clearly the vertical dynamics of $Z$ (i.e.\ $Z$ under the $\Z$ action by $\sigma_{e_d}$) is conjugate to a one-dimensional IUBS, therefore periodic points are dense in it from the case of dimension one. We concude that $p$ appears in a periodic point of $X$ as desired.
\end{proof}

\begin{lemma}
\label{lem:UNSFactor}
Every UNS has an equal entropy full shift factor.
\end{lemma}

\begin{proof}
Let $X$ be our subshift. We may suppose $X$ is nonempty. Let $F : A^N \to \pow(A)$ be the rule, where $N \Subset \Z^{d-1}$. We see $N$ also as a subset of $\Z^d$ by appending a zero to the vectors. Suppose $|F(p)| = k$ for all patterns $p$. For each pattern $p \in A^N$, pick an ordering of $F(p)$, say $F(p) = \{a_{p, 1}, \ldots, a_{p, k}\}$. Now, we can map $x \in X$ to the configuration $y \in \Z_k^{\Z^d}$ defined by
\[ y_{\vec v} = i \iff x_{\vec v} = a_{p, i} \mbox{ where } p = -\vec w + x|(\vec w + N) \]
where $\vec w = \vec v - e_d$ (the role of $\vec w$ is just to translate the pattern below the symbol at $\vec v$ to the correct shape $N$ for applying $F$). Let $\phi$ be the map defined by this formula. It is clearly shift-commuting and continuous. Its image is a subshift of $\Z_k^{\Z^d}$. By Lemma~\ref{lem:Entropy}, it suffices to show surjectivity.

For surjectivity, pick a configuration $y \in \Z_k^{\Z^d}$. 
We construct $x \in X$ such that $\phi(x)|\Z^{d-1} \times \N = y|\Z^{d-1} \times \N$ i.e.\ the top halfspace in the image has arbitrary content.
Pick $S_{-1}(x)$ arbitrarily. Applying the local rule we can pick the contents of the slices $S_i(x)$ for increasing $i$ so that the image configuration $y$ has any desired contents in the top halfspace in its $\phi$-image. We conclude that $\phi(X)$ can contain any pattern over $\Z_k$ in the top halfspace. Since the image is a subshift it must be the full shift.
\end{proof}

\section{Convex avoshifts are inductive binondeterministic spacetimes up to a twist}

For a subshift $Y$, write $Y^{[k]}$ for the $(\{\vec 0^{d-1}\} \times \{0, \ldots, k-1\})$-blocking of $Y$.

\begin{lemma}
Let $X \subset A^{\Z^d}$ be an avoshift for inductive intervals. Then there exist $M \in \GL(d,\Z)$ and $k$ such that $Z = (MX)^{[k]}$ is a binondeterministic spacetime.
\end{lemma}

\begin{proof}
Since $X$ is an avoshift for inductive intervals, its avomanual $\C$ by definition handles inductive interval extensions. By Lemma~\ref{lem:AvomanualPermissive}, $\C$ is also down, monotone up and permissive. Now, Theorem~\ref{thm:MainTechnical} says that if a manual is down, monotone up, permissive, and handles all inductive interval extensions, then it handles in parallel all sufficiently slanted rational halfspace extensions. Therefore, $\C$ handles in parallel all sufficiently slanted rational halfspace extensions.

Let now $\vec w$ be sufficiently slanted. The conclusion is that the avomanual $\C$ contains $(C, D)$ such that $C \Subset H_{\vec w, 0}^\circ$ and $\bar H_{\vec w, 0} \setminus \{\vec 0\} \subset D$. The definition of an avoshift says that for all $x \in X$, we have $\follow_X(x|D, \vec 0) = \follow_X(x|C, \vec 0)$

Let $\bar e_d = \vec u$ be any vector such that $\vec u \cdot \vec w = 1$ (which exists since $\mathrm{gcd}(\vec w_1, \ldots, \vec w_d) = 1$). Pick a free generating set $\bar e_1, \ldots, \bar e_{d-1}$ for the group $\partial H_{\vec w, 0}$ (for this, note that the latter is a torsion-free abelian group and thus isomorphic to some $\Z^k$, and $k = d-1$ is easy to see). It is easy to see that the $\bar e_i$ are then a basis of $\Z^d$. Namely, if $\vec v \in \Z^d$ then $\vec v - (\vec v \cdot \vec w) \vec u \in \partial H_{\vec w, 0}$.

Now let $M$ be the matrix with the $\bar e_i$ as columns, i.e.\ $Me_i = \bar e_i$, and let $Y = M^{-1} X$. Then we have $y_{\vec v} = x_{M \vec v}$. By Lemma~\ref{lem:MatrixAvo}, $(C', D') = (M^{-1}C, M^{-1}D)$ is in the avomanual of $Y$. 

Now we observe that
\[ M^{-1} H_{\vec w, 0}^\circ = H_{e_d, \vec 0}^\circ \mbox{ and } M^{-1} \bar H_{\vec w, 0} = \bar H_{e_d, \vec 0} \]
because $H_{\vec w, 0}^\circ$ (resp.\ $\bar H_{\vec w, 0}$) consists of precisely the vectors $\vec v$ whose canonical representation in the basis $\{\bar e_i\}$ uses a positive coefficient for $\vec u$s (resp.\ nonnegative coefficient).
Therefore,
\[ C' \Subset H_{e_d, 0}^\circ, \bar H_{e_d, 0} \setminus \{\vec 0\} \subset D'. \]
Since $C'$ is finite, $C' \subset H_{e_d, 0}^\circ \cap T_{1, k}$ for some $k$.

Define now $Z = Y^{[k]}$. Then in $Z$, the rule $(C', D')$ gives rise to a rule $(C'', D'')$, where
\[ C' \Subset S_1, \bar H_{e_d, 0} \setminus \{\vec 0\} \subset D'. \]

Now we observe that in all this discussion, we just needed $\vec w$ to be sufficiently slanted and $k$ sufficiently large. Thus we can simultaneously apply the discussion to the inversion of $X$: inductive intervals are invariant under the inversion map $\vec v \mapsto -\vec v$, so the inversion $X'$ of $X$ is an avoshift for inductive intervals. Furthermore, the above construction applied to $X'$ gives precisely the inversion $Z'$ of $Z$ when using the same vector $\vec w$ to define the halfspaces.

We conclude that for $\vec w$ sufficiently slanted and $k$ sufficiently large, $Z$ is an IBS as was to be proved.
\end{proof}


\begin{theorem}
Let $X \subset A^{\Z^d}$ be an avoshift for convex sets. Then there exist $M \in \GL(d,\Z)$ and a finite convex set $N$ such that $Z = (MX)^{[N]}$ is an inductive binondeterministic spacetime.
\end{theorem}

\begin{proof}
We prove this by induction on dimension. By the previous lemma, we can pick a matrix $M$ and a vertical blocking size $k$ so that $Y = (MX)^{[k]}$ is a binondeterministic spacetime. Since $X$ is an avoshift for convex sets, so is $Y$. Thus, we can apply induction to the $\Z^{d-1}$-trace of $Y$. Since the blocking and matrix applied there is orthogonal to $M$ and the $k$-blocking, we can combine the two. This gives the statement.
\end{proof}

From results on nondeterministic spacetimes, we obtain as immediate corollaries the results from the introduction.

\begin{theorem}
Every convex avoshift has a periodic point.
\end{theorem}

\begin{proof}
This follows from Lemma~\ref{lem:IBSPeriodic}.
\end{proof}

\begin{theorem}
Let $X$ be a unishift for convex sets. Then $X$ has dense totally periodic points.
\end{theorem}

\begin{proof}
This follows from Lemma~\ref{lem:IUBSPeriodic}.
\end{proof}

\begin{theorem}
Let $X$ be a unishift for convex sets. Then $X$ has an equal entropy full shift factor.
\end{theorem}

\begin{proof}
This follows from Lemma~\ref{lem:UNSFactor}.
\end{proof}

\section{The natural measure of a unishift}
\label{sec:Measure}

Finally, we explain why unishifts (almost by definition) admit a natural measure.

\begin{theorem}
\label{thm:NaturalMeasure}
Equation~\ref{eq:unimeasure} defines a full-support measure of maximal entropy for unishifts.
\end{theorem}

\begin{proof}
Let $X$ be a unishift on convex sets. The proof of \cite[Theorem~5.5]{Sa22d} shows that Equation~\ref{eq:unimeasure} defines a shift-invariant measure $\mu$ on $X$. Full support is obvious.

For the entropy, since $[0,n-1]^d]$ is convex, $\mu(p) = 1/\# X|[0,n-1]^d$ for all $p \in X|[0,n-1]^d$. Thus
\[ -\mu(p) \log \mu(p) = (\log (\# X|[0,n-1]^d)) / \# X|[0,n-1]^d) \]
so
\[ \frac{\sum_{p \in X|[0,n-1]^d} -\mu(p) \log \mu(p)}{n^d} = \frac{\log \# X|[0, n-1]^d}{n^d}, \]
thus the topological and measure-theoretic entropies coincide.
\end{proof}

\section{Questions}

There is a missing entry in Table~\ref{tab:Properties}:

\begin{question}
Does every avoshift factor onto any full shift of lower (or even equal) entropy?
\end{question}

This is true in dimension $1$ \cite[Theorem 5.5.8]{LiMa95} (where avoshifts are just the SFTs).

In \cite{Sa24} it was shown that it is decidable whether a given SFT is an avoshift for inductive intervals. Here, we assume the avo property on convex sets, and the same proof does not seem to work.

\begin{question}
Given an SFT $X \subset A^{\Z^d}$, is it decidable whether $X$ is avo for convex sets?
\end{question}

In \cite{Sa24}, it was shown that if a subshift is avo for inductive intervals, then it is also SFT on inductive intervals, in the sense that there is a finite set of forbidden patterns such that a locally valid configuration on an inductive interval is also globally valid. We do not have a similar result for avoshifts on convex sets.

\begin{question}
If an SFT $X \subset A^{\Z^d}$ is avo for convex sets, is it SFT on convex sets? 
\end{question}

We have
\[ \mbox{convex-avo} \overset{*}\subset \mbox{IBS} \subset \mbox{II-avo} \]
where II means inductive intervals, and $\overset{*}\subset$ means inclusion up to convex blocking and a base change of $\Z^d$. The first inclusion is the main structure result in this paper, and the second inclusion is not difficult to show (of course, one should use the same ordering of the dimensions). Nevertheless, we feel that we do not have a clear picture of what the connections are between being avo for inductive intervals, being avo for convex sets, and being obtained inductively from nondeterministic cellular automata in the sense of IBS.

The theory in \cite{Sa24} was built more generally for polycyclic groups. However, the theory developed in the present article depends on geometric properties of the group $\Z^d$, and we cannot prove similar results even on the three-dimensional Heisenberg group. 

\begin{question}
Does every avoshift on a polycyclic group have a periodic point?
\end{question}

\begin{question}
Does every unishift on a polycyclic group have dense periodic points?
\end{question}

\begin{question}
Does every unishift on a polycyclic group have an equal entropy full shift factor?
\end{question}

Theorem~\ref{thm:NaturalMeasure} shows that unishifts have a measure of maximal entropy with full support. This measure need not be ergodic, since a disjoint union of two copies of the same unishift (one with renamed symbols) is clearly a unishift, and the measure we obtain is an affine combination of the natural measures of the two pieces.

\begin{question}
Can a unishift have infinitely many ergodic measures of maximal entropy?
\end{question}

\begin{question}
What can be said about shift-invariant measures on avoshifts?
\end{question}

\bibliographystyle{plain}
\bibliography{../../../bib/bib}{}

\end{document}